\documentclass{amsart}
 
\newtheorem{theorem}{Theorem}[section]
\newtheorem{lemma}[theorem]{Lemma}

\newtheorem{proposition}[theorem]{Proposition}
\newtheorem{corollary}[theorem]{Corollary}

\theoremstyle{definition}
\newtheorem{definition}[theorem]{Definition}
\newtheorem{example}[theorem]{Example}
\newtheorem{notation}[theorem]{Notation}

\theoremstyle{remark}

\numberwithin{equation}{section}

\newcommand\style{\mathcal }          %%% calligraphic

\newcommand{\B}{\style{B}}
\newcommand{\M}{\style{M}}

\newcommand\A{{\style A}}
\renewcommand{\H}{\style{H}}

\newcommand\osi{{\style I}}

\newcommand\osq{{\style Q}}
\newcommand\osr{{\style R}}
\newcommand\oss{{\style S}}
\newcommand\ost{{\style T}}

\newcommand\omin{\otimes_{\rm min}}
\newcommand\omax{\otimes_{\rm max}}

\newcommand\oc{\otimes_{\rm c}}
\newcommand\cstar{{\rm C}^*}                              %%% C$^*$-algebra generated by
\newcommand\cstare{{\rm C}_{\rm e}^*}              %%% C$^*$-envelope of
\begin{document}

\title[Categorical Relations and Tensor Cones ]{Categorical Relations and 
Bipartite Entanglement in Tensor Cones for Toeplitz and Fej\'er-Riesz Operator Systems}

\author{Douglas Farenick}
\address{Department of Mathematics and Statistics, University of Regina, Regina, Saskatchewan S4S 0A2, Canada}
\curraddr{}
\email{douglas.farenick@uregina.ca}
\thanks{Supported in part by the NSERC Discovery Grant program}
 

 %%%%%%%%%%%%%%%%%%%%%
\begin{abstract}
The present paper aims to understand separability and entanglement in tensor cones, in the sense of 
Namioka and Phelps \cite{namioka--phelps1969}, that arise from the base cones of operator system tensor 
products \cite{kavruk--paulsen--todorov--tomforde2011}.
Of particular interest here are the Toeplitz and Fej\'er-Riesz operator systems, which are, respectively,
operator systems of Toeplitz matrices and Laurent polynomials of bounded degree (that is,
trigonometric polynomials), and which are related in the operator system category through duality.  
Some additional categorical relationships established in this paper for Toeplitz and Fej\'er-Riesz operator systems.
Of independent interest is 
a single matrix criterion, similar to the criterion involving
the Choi matrix \cite{choi1972}, for a linear map of the Fej\'er-Riesz operator system to be completely positive.
\end{abstract}

\maketitle

%%%%%%%%%%%%%%%%%%%%%%%%% Text of the Paper
 
 \section{Introduction}
 
  %%%%%%%%%%%%%%%%%%
 \subsection{Tensor cones and the category $\mathfrak S_1$}
 A cone $C$ in a finite-dimensional topological real vector space $V$ (that is, a subset $C\subseteq V$ 
such that $\alpha x + \beta y\in C$, for all nonnegative
 $\alpha,\beta\in\mathbb R$ and all $x,y\in C$)
 is \emph{proper} if it is closed, generating ($C-C=V$), and pointed ($C\cap(-C)=\{0\}$).

 The dual cone $C^d$ of a cone $C$ in $V$ is the subset
 \[
 C^d=\left\{\varphi\in V^d\,|\,\varphi(x)\geq0, \,\forall x\in C\right\}
 \]
 of the dual space $V^d$ of $V$. If $C$ is a proper cone, then so is $C^d$. Moreover,
 in identifying the bidual $V^{dd}$ of $V$ with $V$, the bipolar theorem
 yields the identification of $C^{dd}$ with $C$. A proper cone $C$ of a finite-dimensional 
 topological real vector space $V$ induces a partial order on $V$ in which
 $x\leq y$ (or $y\geq x$), for $x,y\in V$, is understood to denote $y-x\in C$. Thus, $C$ coincides with those $x\in V$ for which $x\geq 0$;
 such elements are said to be \emph{positive}.
 
 If $C_1$ and $C_2$ are cones in finite-dimensional real topological vector spaces $V_1$ and $V_2$, then the
 \emph{$(C_1,C_2)$-separability cone} in $V_1\otimes V_2$, which here will simply be called the \emph{separability cone}, 
 is defined to be the set 
 \[
 C_1\otimes_{\rm sep} C_2=\left\{\sum_{j=1}^k a_j\otimes b_j\,|\,k\in\mathbb N, \, a_j\in C_1,\,b_j\in C_2\right\}.
 \]
 The set $C_1\otimes_{\rm sep} C_2$ is a proper cone in $V_1\otimes V_2$, if each $C_i$ is a proper cone on $V_i$.
 The \emph{$(C_1,C_2)$-dual separable cone} in
 $V_1\otimes V_2$, or simply the \emph{dual separable cone}, is the set
 \[
 \begin{array}{rcl}
 C_1\otimes_{{\rm sep}^*} C_2 &=&\left(C_1^d\otimes_{\rm sep} C_2^d\right)^d \\ && \\
 &=&\left\{\xi \in V_1\otimes V_2\,|\,(\varphi_1\otimes \varphi_2)(\xi)\geq0,\,\forall\,\varphi_i\in C_i^d\right\},
 \end{array}
 \]
 and it is a proper cone if $C_1$ and $C_2$ are proper cones. 
 
 It is straightforward to see that $C_1\otimes_{\rm sep} C_2\subseteq 
 C_1\otimes_{{\rm sep}^*} C_2$. Therefore, a fundamental question is, ``Which cones $C_1$ and $C_2$
 satisfy  $C_1\otimes_{\rm sep} C_2= C_1\otimes_{{\rm sep}^*} C_2$?'' This long-standing question 
 stemming from \cite{namioka--phelps1969} has been answered only quite
 recently by Aubrun, Ludovico, Palazuelos, and Pl\'avala
  \cite{aubrun--ludovico--palazuelos2021}, where it is proved that equality holds if and only one of the cones $C_i$ is a simplicial cone.
 (A proper cone $C$ in a finite-dimensional vector space
 $V$ is said to be \emph{simplicial} if $V$ has a linear basis $B$ in which $C$ is the cone generated by $B$.)
 
 \begin{definition} A cone $C\subset V_1\otimes V_2$ is a \emph{tensor cone} for
$C_1$ and $C_2$ if
\[
C_1\otimes_{\rm sep} C_2 \subseteq C \subseteq C_1\otimes_{{\rm sep}^*} C_2.
\]
Elements of $C_1\otimes_{\rm sep} C_2$ are said be \emph{separable}; elements of
 $C\setminus\left(C_1\otimes_{\rm sep} C_2\right)$ are said to be
 \emph{entangled}.
\end{definition}

Thus, in cases where neither $C_1$ nor $C_2$ is classical, one of the main tasks in the analysis of a given tensor
cone $C$ for proper cones $C_1$ and $C_2$ is to discern which elements of $C$ are separable and which are entangled.
Analyses of this type have been carried out, for example, in
the recent works \cite{aubrun--muller-hermes2023,bluhm--jencova--nechita2022}, 
providing a mathematical framework for general physical theories other than the purely classical or quantum settings.
In the present paper, 
the tensor cones under study are those that arise as the base positive cone of an operator system tensor product.

The terminology and notation for $C_1\otimes_{\rm sep} C_2$ and $C_1\otimes_{{\rm sep}^*} C_2$ above
 differs from some of the standardly used terminology and notation, namely ``min'' and ``max'' and $\omin$ and $\omax$
 \cite{aubrun--ludovico--palazuelos2021,aubrun--muller-hermes2023,bluhm--jencova--nechita2022,fritz2012b}. The reason
 for these differences is because the symbols  $\omin$ and $\omax$ and terms ``min'' and ``max''
 will be reserved in the present paper to reference the minimal and maximal operator system tensor 
 products \cite{kavruk--paulsen--todorov--tomforde2011}.

%%%%%%%%%%%%%%%
\subsection{Operator systems}

Formally, an
\emph{operator system} is a triple $(\osr, \{\mathcal C_n\}_{n\in\mathbb N}, e_\osr)$ consisting of a complex $*$-vector
space $\osr$, a family $\{\mathcal C_n\}_{n\in\mathbb N}$ of proper cones in the real vector spaces $\M_n(\osr)_{\rm sa}$
satisfying $\alpha^*\mathcal C_n \alpha\subseteq \mathcal C_m$, for all $n,m\in\mathbb N$ and $n\times m$ complex matrices $\alpha$,
and a distinguished element $e_\osr\in \mathcal C_1$ that serves as an Archimedean order unit for the 
family $\{\mathcal C_n\}_{n\in\mathbb N}$ \cite{choi--effros1977,Paulsen-book}. 
Here, the notation $\mathcal V_{\rm sa}$, for a complex $*$-vector space $\mathcal V$, 
means the real vector space $\mathcal V_{\rm sa}=\{v\in\mathcal V\,|\,v^*=v\}$ of self adjoint elements. (With matrices over
$\osr$, the induced adjoint operation is $([x_{ij}]_{i,j=1}^n)^*=[x_{ji}^*]_{i,j=1}^n$.)
It is common to dispense with the triple notation, and simply
refer to $\osr$ as an operator system and denote each proper cone $\mathcal C_n$ by $\M_n(\osr)_+$. The cone $\mathcal C_1$
is called the \emph{base (positive) cone} of $\osr$, and is denoted by $\osr_+$, while the \emph{matrix cones} $\mathcal C_n$ are denoted
by $\M_n(\osr)_+$, for each $n$.

Linear transformations $\phi:\osr\rightarrow\ost$ of operator systems are said to be completely positive if their ampliations
$\phi^{(n)}:\M_n(\osr)\rightarrow\M_n(\ost)$, which are
defined by $\phi^{(n)}([x_{ij}]_{i,j=1}^n)=[\phi(x_ij)]_{i,j=1}^n$,
map $\M_n(\osr)_+$ in $\M_n(\ost)_+$, for each $n$, and these maps are unital if $\phi(e_\osr)=e_\ost$.

By $\mathfrak S_1$ we denote the category whose 
objects are operator systems and morphisms are unital completely positive (ucp) linear maps. 
Two operator systems, $\osr$ and $\ost$, are isomorphic in the category $\mathfrak S_1$ if there is linear bijection
$\phi:\osr\rightarrow\ost$ such that both $\phi$ and $\phi^{-1}$ are ucp maps; such bijections
$\phi$ are called unital complete order isomorphisms and the notation $\osr\simeq\ost$ is used to
denote that $\osr$ and $\ost$ are isomorphic in $\mathfrak S_1$. Unital C$^*$-algebras are among the objects in $\mathfrak S_1$,
and the theory of operator systems %\cite{choi--effros1977,Paulsen-book}
is, in general, concerned with unital $*$-closed subspaces of 
unital C$^*$-algebras, as the Choi-Effros Embedding Theorem \cite{choi--effros1977}
shows that for every operator system $\osr$ there
are a Hilbert space $\H$ and an operator subsystem $\ost$
of the algebra $\B(\H)$ of bounded linear operators on $\H$
such that $\osr\simeq\ost$.

The category $\mathfrak S_1$ admits a variety of constructions, including quotients and tensor products.
With respect to tensor products,
there are maximal and minimal operator system tensor product structures, 
whereby the maximal tensor product has the smallest family of matricial
cones, while the minimal tensor product has the largest family of positive cones
\cite{kavruk--paulsen--todorov--tomforde2011}. 
Matricial cones are used to norm elements of operator systems; 
thus, small cones lead to large [e.g., ``max''] norms, while large cones lead to small [e.g., ``min''] norms.
 
 %%%%%%%%%%%%%%%%%%%%%%%%%%%%%
 \subsection{Real versus complex vector spaces}
 
 Operator systems are complex $*$-vector spaces, but the theory of tensor cones is framed in terms of real vector spaces; therefore, 
 it is important to be clear about the definitions of the structures involved.
 
 The dual space $\osr^d$ of a finite-dimensional operator system
 $\osr$ is also an operator system \cite{choi--effros1977}, and so one can consider operator system tensor products of the form 
 $\osr^d\otimes_\sigma\ost$ for finite-dimensional operator systems
 $\osr$ and arbitrary operator systems $\ost$. At the base level, the base cone $(\osr^d)_+$ is given by all 
 linear functionals $\phi:\osr\rightarrow\mathbb C$ such that $\phi(y)\geq0$, for all $y\in\osr_+$.
 
 Consider now both the complex $*$-vector space $\osr$ and the real vector space $\osr_{\rm sa}$. If
 $\phi:\osr\rightarrow\mathbb C$ is a linear functional such that $\phi(y)\geq0$, for all $y\in\osr_+$, then necessarily
 $\phi(x^*)=\overline{\phi(x)}$, for every $x\in\osr$; thus, $\phi$ is also a linear functional $\osr_{\rm sa}\rightarrow\mathbb R$
 and is an element of the dual of $\osr_+$ when considered as a proper cone in $\osr_{\rm sa}$. Conversely,
 if $\vartheta:\osr_{\rm sa}\rightarrow\mathbb R$ is a linear functional for which $\vartheta(x)\geq0$
 for every $x\in\osr_+$, then the linear map $\phi:\osr\rightarrow\mathbb C$ defined by
 $\phi(x+iy)=\vartheta(x)+i\vartheta(y)$, for $x,y\in\osr_{\rm sa}$, defines a complex linear functional
 for which $\phi(x)\geq0$ when $x\in\osr_+$. Thus, the dual of the cone $\osr_+$,
 when considered as a proper cone in $\osr_{\rm sa}$, is canonically identified with the base cone $(\osr^d)_+$ of the
 dual operator system $\osr^d$.

 Using the notation $\otimes_{\mathbb R}$ and $\otimes_{\mathbb C}$ to distinguish the tensor
 product operations in the categories of real and complex vector spaces, respectively, one needs to understand 
 how $\mathcal V_{\rm sa} \otimes_{\mathbb R} \mathcal W_{\rm sa} $ relates to 
 $\left(\mathcal V\otimes_{\mathbb C} \mathcal W \right)_{\rm sa}$ when $\mathcal V$ and $\mathcal W$ are
complex $*$-vector spaces. In general,
$\mathcal V_{\rm sa} \otimes_{\mathbb R} \mathcal W_{\rm sa} \subseteq
\left(\mathcal V\otimes_{\mathbb C} \mathcal W \right)_{\rm sa}$; however, if
$\mathcal W$ is the algebra $\M_n(\mathbb C)$ of complex $n\times n$ matrices, 
then
$\mathcal V_{\rm sa} \otimes_{\mathbb R} \mathcal W_{\rm sa} =\left(\mathcal V\otimes_{\mathbb C} \mathcal W \right)_{\rm sa}$
\cite[Lemma 3.7]{paulsen--todorov--tomforde2010}.

Returning to the case of finite-dimensional operator systems, $\osr$ and $\ost$, 
there are two possible field-dependent
definitions of 
$\osr_+\otimes_{{\rm sep}^*}\ost_+$, and we distinguish them below by using a superscript to denote the base field:
\[
\osr_+\otimes_{{\rm sep}^*}^{\mathbb R}\ost_+
=
\left\{ x\in \osr_{\rm sa} \otimes_{\mathbb R} \ost_{\rm sa}\,|\, (\phi\otimes\psi)[x]\geq0,\mbox{ for all }\phi\in(\osr_{\rm sa}^d)_+,\,
\psi\in(\ost_{\rm sa}^d)_+\right\},
\]
and
\[
\osr_+\otimes_{{\rm sep}^*}^{\mathbb C}\ost_+
=
\left\{ x\in \left(\osr  \otimes_{\mathbb C} \ost\right)_{\rm sa}\,|\, (\phi\otimes\psi)[x]\geq0,\mbox{ for all }\phi\in(\osr^d)_+,\,
\psi\in(\ost ^d)_+\right\}.
\]
As explained above, $(\osr_{\rm sa}^d)_+$ and $(\osr^d)_+$ are canonically identified, and so
\[
\osr_+\otimes_{{\rm sep}^*}^{\mathbb R}\ost_+
\subseteq
\osr_+\otimes_{{\rm sep}^*}^{\mathbb C}\ost_+
\]
in general.

Therefore, to accommodate our interest in complex $*$-vector spaces,
the following notation shall be used henceforth.

\begin{notation}\label{notation cmplx} If $\osr$ and $\ost$ are finite-dimensional operator systems, then 
\[
\osr_+\otimes_{{\rm sep}^*}\ost_+
=
\left\{ x\in \left(\osr  \otimes_{\mathbb C} \ost\right)_{\rm sa}\,|\, (\phi\otimes\psi)[x]\geq0,\mbox{ for all }\phi\in(\osr^d)_+,\,
\psi\in(\ost ^d)_+\right\}.
\]
That is, $\osr_+\otimes_{{\rm sep}^*}\ost_+$ denotes $\osr_+\otimes_{{\rm sep}^*}^{\mathbb C}\ost_+$.
\end{notation}

With the notation above, the main result of \cite{aubrun--ludovico--palazuelos2021} still holds.

\begin{theorem}\label{equality}
If $\osr$ and $\ost$ are finite-dimensional operator systems, then 
\[
\osr_+\otimes_{{\rm sep}}\ost_+
=
\osr_+\otimes_{{\rm sep}^*}\ost_+
\]
if and only if one of the cones $\osr_+$ or $\ost_+$ is simplicial.
\end{theorem}

\begin{proof}
If $\osr_+\otimes_{{\rm sep}}\ost_+=\osr_+\otimes_{{\rm sep}^*}\ost_+$, then it is also true that
$\osr_+\otimes_{{\rm sep}}\ost_+=\osr_+\otimes_{{\rm sep}^*}^{\mathbb R}\ost_+$; hence, 
by the main result of \cite{aubrun--ludovico--palazuelos2021}, one of 
$\osr_+$ or $\ost_+$ is simplicial.

Conversely, assume $\osr_+$ is simplicial. Thus, $\osr_{\rm sa}$ has a linear basis $\{e_1,\dots,e_d\}$ consisting
of positive $e_j$ such that $\osr_+$ is the cone generated by $\{e_1,\dots,e_d\}$. Hence, each $x\in \osr\otimes_{\mathbb C}\ost$
has the form $x=\displaystyle\sum_{j=1}^d e_j\otimes t_j$, for some $t_j\in\ost$. If, further, $x^*=x$, then $t_j^*=t_j$ by the
linear independence of the $e_j$, which yields $x\in \osr_{\rm sa}\otimes_{\mathbb R}\ost_{\rm sa}$. In other words,
$\osr_{\rm sa}\otimes_{\mathbb R}\ost_{\rm sa}=(\osr\otimes_{\mathbb C}\ost)_{\rm sa}$. Hence, by the main result of 
\cite{aubrun--ludovico--palazuelos2021}, the equality 
\[
\osr_+\otimes_{{\rm sep}}\ost_+=\osr_+\otimes_{{\rm sep}^*}^{\mathbb R}\ost_+
\]
holds, which yields $\osr_+\otimes_{{\rm sep}}\ost_+
=
\osr_+\otimes_{{\rm sep}^*}\ost_+$ because 
$\osr_{\rm sa}\otimes_{\mathbb R}\ost_{\rm sa}=(\osr\otimes_{\mathbb C}\ost)_{\rm sa}$.
\end{proof}

With Notation \ref{notation cmplx} in mind, the following example 
is one of the most important convex cones in quantum theory, consisting of what are called
\emph{block-positive matrices}, which are employed as entanglement witnesses.

\begin{example}[Block-positive matrices] For every $n\in\mathbb N$ and finite-dimensional operator system $\ost$,
\[
\ost_+\otimes_{{\rm sep}^*}\M_n(\mathbb C)_+=\left\{[x_{ij}]_{i,j=1}^n\in \M_n(\ost)_{\rm sa}\,|\,
\sum_{i,j=1}^n \alpha_i\overline\alpha_j x_{ij}\in \ost_+, \mbox{ for all }\alpha_j\in\mathbb C\right\}.
\]
\end{example}

\begin{proof} The assertion is Theorem 3.2 of \cite{paulsen--todorov--tomforde2010}.
\end{proof}

%%%%%%%%%%%%%%%%%%%%%%%%%%%%%%%%
\subsection{Tensor cones with operator systems}

As will be explained in Proposition \ref{tc}, and
using Notation \ref{notation cmplx},
if both $\osr$ and $\ost$ are finite-dimensional operator systems, 
then 
\begin{equation}\label{e:incl}
\osr_+ \otimes_{\rm sep} \ost_+ \subseteq \left(\osr\omax\ost\right)_+ \subseteq \left(\osr\otimes_\sigma\ost\right)_+
\subseteq \left(\osr\omin\ost\right)_+ 
\subseteq \osr_+ \otimes_{{\rm sep}^*} \ost_+.
\end{equation}
In other words, the base cone $(\osr\otimes_\sigma\ost)_+$ is a tensor cone for $\osr_+$ and $\ost_+$.

In light of the main result of \cite{aubrun--ludovico--palazuelos2021} (in the formulation of Theorem \ref{equality}), 
equality across \eqref{e:incl} will occur only
if one of $\osr_+$ or $\ost_+$ is a simplicial cone. However, it is interesting to have a better
understanding of the other inclusions in \eqref{e:incl}, and the first purpose of the present paper is to address this
when $\osr$ and $\ost$ are drawn from Toeplitz and Fej\'er-Riesz operator systems.

 %%%%%%%%%%%%%%%%%%%%%%%%%%%%%
 \subsection{Toeplitz and Fej\'er-Riesz operator systems}

 A \emph{Toeplitz operator system} is an operator 
 system of $n\times n$ complex matrices, for some $n\geq 2$, in which the matrices are
 of Toeplitz form \cite{toeplitz1911}:
\begin{equation}\label{tms-intro}
 \left[ \begin{array}{ccccc} \alpha_0 & \alpha_{-1}  & \alpha_{-2} & \dots & \alpha_{-n+1}  \\
\alpha_1 & \alpha_0 & \alpha_{-1} &\ddots & \vdots \\ 
\alpha_2 & \alpha_1 &   \alpha_0  &\ddots & \alpha_{-2} \\ 
\vdots & \ddots & \ddots & \ddots & \alpha_{-1} \\
\alpha_{n-1} & \dots & \alpha_2 &\alpha_1 & \alpha_0 
\end{array}
\right],
\end{equation}
for some $\alpha_\ell\in\mathbb C$.
The linear space $C(S^1)^{(n)}$ of all such $n\times n$ complex matrices
forms an operator subsystem, denoted by $C(S^1)^{(n)}$, of the unital
C$^*$-algebra $\M_n(\mathbb C)$ of $n\times n$ complex matrices, where the $n\times n$. Thus, a Toeplitz matrix $x$
is positive if it is positive as linear operator on the Hilbert space $\mathbb C^n$, 
and the
identity matrix $1_n$ serves as the canonical Archimedean order unit for $\M_n(\mathbb C)$ and, hence,
$C(S^1)^{(n)}$. At the matrix level, $x\in \M_p(C(S^1)^{(n)})$ is positive if $x$ is positive as an operator on the
Hilbert space $\mathbb C^n\otimes\mathbb C^p\cong \displaystyle\bigoplus_1^p \mathbb C^n$.

A \emph{Fej\'er-Riesz operator system} is an operator subsystem of the unital abelian 
C$^*$-algebra $C(S^1)$ of continuous functions 
 $f:S^1\rightarrow\mathbb C$ on the unit circle $S^1\subset\mathbb C$. If $f\in C(S^1)$, then the Fourier coefficients  
$\hat f(k)$ of $f$ are given by
 \[
 \hat f(k)=\frac{1}{2\pi}\int_0^{2\pi} f(e^{i\theta})e^{-i k\theta}\,d\theta,
 \]
for every $k\in\mathbb Z$. If $n\geq2$ is fixed, then Fej\'er-Riesz operator system
$C(S^1)_{(n)}$ is defined to be the set of those $f\in C(S^1)$ for which $\hat f(k)=0$, 
for all $k\in\mathbb Z$ with $|k|\geq n$. 
Thus, each  $f\in C(S^1)_{(n)}$ has the form
 \[
 f(z)=\sum_{\ell=-n+1}^{n-1} \alpha_\ell z^\ell,
 \]
 for some $\alpha_\ell\in\mathbb C$. 
 The positive elements of $C(S^1)$ are those continuous functions for which $f(z)\geq0$, for every $z\in S^1$. At the matrix
 level, an element $F\in\M_p\left( C(S^1)_{(n)}\right)$ is positive if, as a matrix-valued function $F:S^1\rightarrow\M_p(\mathbb C)$,
 $F(z)\in \M_p(\mathbb C)_+$, for every $z\in S^1$.
 The constant function $\chi_0$, given by $\chi_0(z)=1$ for all $z\in S^1$, serves as
 the canonical Archimedean order unit for both $C(S^1)$ and $C(S^1)_{(n)}$.
 The term ``Fej\'er-Riesz operator system'' is used because of the importance 
 of the Fej\'er-Riesz factorisation theorem \cite{fejer2015} in
 the study and application of trigonometric polynomials with nonnegative values.

Even though Toeplitz and Fej\'er-Riesz operator systems have been studied intensely since their introduction in 
classical works devoted quadratic forms and Fourier series, 
the following categorical relationship between these operator 
systems is rather recent \cite{connes-vansuijlekom2021,farenick2021}.

\begin{theorem}[Duality]\label{toeplitz duality}
The  linear map $\delta:C(S^1)^{(n)}\rightarrow\left(C(S^1)_{(n)}\right)^d$ that sends a Toeplitz matrix
$T=[\tau_{k-\ell }]_{k,\ell=0}^{n-1}\in C(S^1)^{(n)}$ to the 
linear functional $\varphi_T:C(S^1)_{(n)} \rightarrow\mathbb C$ defined by
\begin{equation}\label{lf defn}
\varphi_T(f)=\sum_{k=-n+1}^{n-1}\tau_{-k}\hat f(k),
\end{equation}
for $f\in C(S^1)_{(n)}$, is a unital complete order isomorphism. 
That is, 
\[
C(S^1)^{(n)}\simeq  \left(C(S^1)_{(n)}\right)^d \,\mbox{ and }\, \left(C(S^1)^{(n)}\right)^d\simeq  C(S^1)_{(n)}
\]
in the operator system category $\mathfrak S_1$.
\end{theorem}

%%%%%%%%%%%%%%
\subsection{Statement of the main results}

%%%%%%%%%
\subsubsection{Base cones of operator system tensor products are tensor cones}

 \begin{theorem} If $\otimes_\sigma$ is an operator system tensor product structure on finite-dimensional
 operator systems $\osr$ and $\ost$, then the base positive cone $(\osr\otimes_\sigma\ost)_+$ of the
 operator system $\osr\otimes_\sigma\ost$ satisfies the inclusions
 \[
 \osr_+\otimes_{\rm sep}\ost_+ \subseteq (\osr\otimes_\sigma\ost)_+ \subseteq \osr_+ \otimes_{{\rm sep}^*} \ost_+.
 \]
 In other words, $(\osr\otimes_\sigma\ost)_+$ is a tensor cone for $\osr_+$ and $\ost_+$.
 \end{theorem}

%%%%
\subsubsection{Separability}
In the operator system category $\mathfrak S_1$, if $\osr$ and $\ost$ are both
Toeplitz or both Fej\'er-Riesz operator systems, then, by the results in \cite{farenick2021},
\[
\osr\omin\ost \not= \osr\omax\ost.
\]
What this means is that the sequence of positive matrix cones for these tensor product operator systems do not coincide
somewhere along the sequence. Nevertheless, equality is sometimes possible at the base level,
$(\osr\otimes_\sigma\ost)_+$,  in an operator system tensor product $\osr\otimes_\sigma\ost$, as demonstrated by the theorem below.

\begin{theorem} Assume that $n,m\in\mathbb N$ satisfy $n,m\geq 2$.
\begin{enumerate}
\item The positive cone of $C(S^1)^{(n)}\omax C(S^1)^{(m)}$ coincides with separability cone 
$(C(S^1)^{(n)})_+\otimes_{\rm sep} (C(S^1)^{(m)})_+$.
\item The positive cone of $C(S^1)^{(2)}\omin C(S^1)^{(m)}$ coincides with separability cone 
$(C(S^1)^{(2)})_+\otimes_{\rm sep} (C(S^1)^{(m)})_+$.
\item If $x_0,x_1$ are $m\times m$ Toeplitz matrices for which
$\left[ \begin{array}{cc} x_0 & x_1^* \\ x_1 & x_0\end{array}\right]$ is a positive operator,
then for every $n\geq 3$ there exist Toeplitz matrices $x_2$,\dots, $x_{n-1}$ such that
\[
\left[ \begin{array}{cccc} x_0 & x_{1}^* & \dots & x_{n-1}^* \\
x_1 & x_0 & \ddots & \vdots \\ \vdots & \ddots & \ddots & x_1^* \\
x_{n-1} & \dots & x_1 & x_0 
\end{array}
\right]
\]
is a positive operator.
\end{enumerate}
\end{theorem}
 
 The first two assertions above sharpen Proposition 3.4 and Corollary 3.5, respectively, in 
 \cite{farenick--mcburney2023}, as well as some results of Ando in \cite{ando2004,ando2013},
 while the third assertion fails if one aims to extend $3\times 3$ positive block-Toeplitz matrices with $m\times m$ Toeplitz blocks to  
 $n\times n$ ($n\geq 4$) positive block-Toeplitz matrices with $m\times m$ Toeplitz blocks.
 
 In addition, a purely algebraic proof establishes the following theorem, which can also be proved using the
 main result of \cite{aubrun--ludovico--palazuelos2021}.
 
 \begin{theorem}\label{btcb} If $\cstar(\mathbb Z_m)$ denotes the operator system of $m\times m$ complex circulant matrices,
 then,  for every $n\geq 2$, 
 \[
 (C(S^1)^{(n)})_+\otimes_{\rm sep} \cstar(\mathbb Z_m)_+
 =
 (C(S^1)^{(n)})_+\otimes_{{\rm sep}^*} \cstar(\mathbb Z_m)_+.
\]
\end{theorem}

%%%%
\subsubsection{Entanglement}

The first result is in contrast to Theorem \ref{btcb} above. 

\begin{theorem} If $\osr_+$ and $\ost_+$ are Toeplitz or Fej\'er-Riesz cones, then
\[
\osr_+\otimes_{\rm sep}\ost_+\not= \osr_+\otimes_{{\rm sep}^*} \ost_+.
\]
\end{theorem}

An important matrix in the study of Toeplitz and Fej\'er-Riesz operator systems is the universal 
positive  $n\times n$ Toeplitz matrix $T_n$, the element of the positive cone of 
$C(S^1)^{(n)}\otimes C(S^1)_{(n)}$  defined by 
\[
T_n(z)
= 
\left[ \begin{array}{cccc} 1 & z^{-1} & \dots & z^{-n+1} \\
z & 1 & \ddots & \vdots \\ \vdots & \ddots & \ddots & z^{-1} \\
z^{n-1} & \dots & z & 1 
\end{array}
\right],
\]
for $z\in S^1$.

\begin{theorem} For every $n\geq 2$, $T_n$ generates an extremal ray of the positive cone of
$C(S^1)^{(n)}\omin C(S^1)_{(n)}$ and is entangled.
\end{theorem}

Interestingly, the matrix $T_n$ has the same role for determining the complete positivity of linear maps on $C(S^1)_{(n)}$ 
that the Choi matrix \cite{choi1972} has for determining the complete positivity of linear maps on $\M_n(\mathbb C)$,
as noted below.

\begin{theorem} A linear map $\phi:C(S^1)_{(n)}\rightarrow\ost$, for an operator system $\ost$, is completely positive if and only
if $\phi^{(n)}(T_n)$ is positive in $\M_n(\ost)$, where $\phi^{(n)}$ is the ampliation of $\phi$ to a linear map
$\M_n(C(S^1)_{(n)})\rightarrow\M_n(\ost)$.
\end{theorem}

Elements of $C(S^1)_{(n)}\otimes C(S^1)_{(m)}$ may viewed naturally as continuous 
complex-valued functions on the 2-torus $S^1\times S^1$. In this regard, we have:

\begin{theorem} For every $n,m\geq 2$,
\[
\left(C(S^1)_{(n)}\omin C(S^1)_{(m)}\right)_+
=
(C(S^1)_{(n)})_+\otimes_{{\rm sep}^*} (C(S^1)_{(m)})_+ 
\]
and 
\[
\left(C(S^1)_{(2)}\omax C(S^1)_{(m)}\right)_+
=
(C(S^1)_{(2)})_+\otimes_{{\rm sep}^*} (C(S^1)_{(m)})_+ .
\]
\end{theorem}

%%%%%%%%
\subsubsection{Categorical relations}
 
\begin{theorem}\label{zxc} If $\osr$ is a Toeplitz or Fej\'er-Riesz operator system, then:
\begin{enumerate}
\item there are no unital C$^*$-algebras $\A$ for which $\osr\simeq \A$;
\item $\osr$ does not have the weak expectation property;
\item  $\osr\omin \B(\H) = \osr\omax\B(\H)$, for every Hilbert space $\H$;
\item $\osr\omin \osi = \osr\omax\osi$, for every injective operator system $\osi$.
\end{enumerate}
\end{theorem}  

The truncation of $m\times m$ Toeplitz matrices to $n\times n$ Toeplitz matrices, when $m\geq n$, is a very natural map
sending large Toeplitz matrices to smaller ones; it is in fact
a complete quotient map in the sense of \cite{farenick--paulsen2012}.

\begin{theorem} $C(S^1)^{(n)}$ is an operator system quotient of $C(S^1)^{(m)}$, if $m\geq n$.
\end{theorem}

%%%%%%%%%%%%%%%%%%%%%%%%%%%%%
\section{Operator System Duality and Tensor Products}

As mentioned earlier, an operator system is a 
triple $(\osr, \{\mathcal C_n\}_{n\in\mathbb N}, e_\osr)$ consisting of a complex $*$-vector
space $\osr$, a family $\{\mathcal C_n\}_{n\in\mathbb N}$ of proper cones in the real vector spaces $\M_n(\osr)_{\rm sa}$
of selfadjoint matrices over $\osr$ in which the cones
satisfy $\alpha^*\mathcal C_n \alpha\subseteq \mathcal C_m$, for all $n,m\in\mathbb N$ 
and $n\times m$ complex matrices $\alpha$,
and a distinguished element $e_\osr\in \mathcal C_1$ that serves as an Archimedean order unit for the 
family $\{\mathcal C_n\}_{n\in\mathbb N}$ \cite{choi--effros1977,Paulsen-book}. Most often the cones $\mathcal C_n$
are denoted by $\M_n(\osr)_+$, with $\osr_+$ denoting $\mathcal C_1$, and $\osr$ is used to designate
the triple $(\osr, \{\mathcal C_n\}_{n\in\mathbb N}, e_\osr)$.

\begin{definition} An operator subsystem $\osr$ of an operator system $\oss$ is a unital linear subspace 
$\osr\subseteq\oss$, closed under the involution
of $\oss$, such that $\M_n(\osr)_+$ is defined to be
$\M_n(\osr)_+=\M_n(\osr)\cap\M_n(\oss)_+$, for every $n\in\mathbb N$, and $e_\osr=e_\oss$. 
\end{definition}

A \emph{unital linear complete order embedding} of an operator system $\osr$ into an operator system $\ost$ is a linear
injection $\phi:\osr\rightarrow\ost$ such that $\phi$ and the inverse map $\phi(\osr)\rightarrow\osr$ of
operator systems
are completely positive. If such a map $\phi$ is also surjective, then $\phi$ is a
\emph{unital linear complete order isomorphism} and we denote this relationship between $\osr$ and $\ost$ by
$\osr\simeq\ost$. We this notation, if $\osr$ and $\ost$ are operator systems and if $\phi:\osr\rightarrow\ost$
is a unital linear complete order embedding, then $\osr\simeq\phi(\osr)$.

If $\osr$ is a finite-dimensional operator system with dual space $\osr^d$, then define an $n\times n$
 matrix $\Phi=[\varphi_{ij}]_{i,j=1}^n$ of linear functionals $\varphi_{ij}\in\osr^d$ on $\osr$ to be positive if
 the linear map $\hat\Phi:\osr\rightarrow\M_n(\mathbb C)$, where $\hat\Phi(x)=[\varphi_{ij}(x)]_{i,j=1}^n$,
 is completely positive. The collection $\{\M_n(\osr^d)_+\}_{n\in\mathbb N}$ satisfies the 
 compatibility requirements for the matrix cones of an operator system.
 In selecting any faithful positive linear functional $\delta$ on $\osr$--by which is meant
 a positive linear functional $\delta$ for which $\delta(y)=0$, for $y\in\osr_+$, occurs only with $y=0$--an Archimedean
 order unit for the matrix ordering $\{\M_n(\osr^d)_+\}_{n\in\mathbb N}$ on $\osr^d$ is obtained, 
 providing $\osr^d$ with the structure of an operator system \cite{choi--effros1977}.
 
Turning to tensor products
 \cite{kavruk--paulsen--todorov--tomforde2011,kavruk--paulsen--todorov--tomforde2013},
if $(\osr, \{\mathcal P_n\}_{n\in\mathbb N}, e_\osr)$ and 
$(\ost, \{\mathcal Q_n\}_{n\in\mathbb N}, e_\ost)$ are operator systems, 
then an operator system tensor product structure on the algebraic tensor product
$\osr\otimes\ost$ is a family $\sigma=\{\mathcal C_n\}_{n\in\mathbb N}$ of 
cones $\mathcal C_n\subseteq \M_n(\osr\otimes\ost)$ such that:
\begin{enumerate}
\item[{(i)}] $(\osr\otimes\ost, \{\mathcal C_n\}_{n\in\mathbb N}, e_\osr\otimes e_\ost)$ is an operator system, denoted by $\osr\otimes_\sigma\ost$; 
\item[{(ii)}] $\mathcal P_{n_1}\otimes_{\rm sep} \mathcal Q_{n_2}\subseteq \mathcal C_{n_1n_2}$, for all $n_1,n_2\in\mathbb N$; and
\item[{(iii)}] for all $n_1,n_2\in\mathbb N$, the linear map 
$\phi\otimes\psi:\osr\otimes_\sigma\ost\rightarrow \M_{n_1}(\mathbb C)\otimes\M_{n_2}(\mathbb C)$ is a ucp map, 
whenever $\phi:\osr\rightarrow\M_{n_1}(\mathbb C)$ and $\psi:\ost\rightarrow\M_{n_2}(\mathbb C)$ are ucp maps.
\end{enumerate}

In item (iii) above, $\M_{n_1}(\mathbb C)\otimes\M_{n_2}(\mathbb C)$ is the 
unique C$^*$-algebra tensor product structure on the tensor product of the
unital C$^*$-algebras $\M_{n_1}(\mathbb C)$ and $\M_{n_2}(\mathbb C)$.

If $\otimes_\sigma$ and $\otimes_\tau$ are operator system tensor product 
structures on $\osr\otimes\ost$, then the notation
\[
\osr\otimes_\sigma\ost \subseteq_{{}_{+}} \osr\otimes_\tau\ost
\]
is used to indicate that the linear identity map $\iota(x)=x$ is a ucp map when considered as a linear map
$\iota:\osr\otimes_\sigma\ost \rightarrow \osr\otimes_\tau\ost$ of operator systems. If, in addition, the ucp map
$\iota$ is unital complete order isomorphism (that is, if the inverse of $\iota$ is completely positive), then 
we write
\[
\osr\otimes_\sigma\ost = \osr\otimes_\tau\ost.
\]

If $\osr_1$ and $\ost_1$ are operator subsystems of operator systems $\osr_2$ and $\ost_2$, and if 
$\otimes_\sigma$ and $\otimes_\tau$ are operator system tensor product structures on $\osr_1\otimes\ost_1$ and 
$\osr_2\otimes\ost_2$, respectively, then the notation
\[
\osr_1\otimes_\sigma\ost_1 \subseteq_{{}_{+}} \osr_2\otimes_\tau\ost_2
\]
is used to indicate that the canonical embedding $x\mapsto x$ is a ucp map when considered as a linear map
$\osr_1\otimes_\sigma\ost_1 \rightarrow\osr_2\otimes_\tau\ost_2$ of operator systems. If this ucp embedding is also a
complete order embedding, by which is meant that a matrix $x\in\M_p(\osr_1\otimes \ost_1)$ belongs to $\M_p(\osr_1\otimes_\sigma\ost_1)_+$
if and only if $x$ is an element of $\M_p(\osr_2\otimes_\tau\ost_2)_+$, for every $p\in\mathbb N$, then this situation is denoted by
\[
\osr_1\otimes_\sigma\ost_1 \subseteq_{\rm coi} \osr_2\otimes_\tau\ost_2.
\]

The following definitions were introduced in \cite{kavruk--paulsen--todorov--tomforde2011}.

\begin{definition} 
The \emph{minimal operator system tensor product}, $\omin$, 
of operator systems $\osr$ and $\ost$ is the operator system tensor product structure on the algebraic tensor product
$\osr\otimes\ost$ that is obtained by
declaring a matrix $x\in\M_p(\osr\otimes\ost)$ to be positive 
if $(\phi\otimes\psi)^{(p)}[x]$ is a positive element of $\M_p\left(\M_n(\mathbb C)\otimes\M_q(\mathbb C)\right)$, for every
$n,q\in\mathbb N$ and unital completely positive linear maps $\phi:\osr\rightarrow\M_n(\mathbb C)$
and $\psi:\ost\rightarrow\M_q(\mathbb C)$.
\end{definition}

\begin{definition}
The \emph{maximal operator system tensor product}, $\omax$, 
of operator systems $\osr$ and $\ost$ is the operator system tensor product structure 
on the algebraic tensor product $\osr\otimes\ost$ that is obtained by
declaring a matrix $x\in\M_p(\osr\otimes\ost)$ to be positive 
if, for each $\varepsilon>0$, there are $n,q\in\mathbb N$, $a\in \M_n(\osr)_+$, $b\in\M_q(\ost)_+$, and a linear map 
$\delta:\mathbb C^p\rightarrow\mathbb C^n\otimes\mathbb C^q$
such that
\[
\varepsilon(e_\osr\otimes e_\ost) + x= \delta^*(a\otimes b)\delta.
\]
\end{definition}

With respect to the notation established above, we have the following relationships, for all operator systems $\osr$ and $\ost$, and all 
operator system tensor product structures $\otimes_\sigma$ on $\osr\otimes\ost$:
\begin{equation}\label{e:ten rel}
\osr\omax\ost \subseteq_{{}_{+}} \osr\otimes_\sigma\ost\subseteq_{{}_{+}} \osr\omin\ost.
\end{equation}
Furthermore, as noted in \cite{kavruk--paulsen--todorov--tomforde2011}, if unital C$^*$-algebras $\A$ and $\B$ are 
considered as operator systems, then 
the operator system tensor products $\A\omin\B$ and $\A\omax\B$ are unitally completely order isomorphic 
to the image of $\A\otimes\B$ inside the minimal and
maximal C$^*$-algebraic tensor products of $\A$ and $\B$, respectively.

An alternative to the defining condition for membership in $(\osr\omax\ost)_+$ is given by the next result.

\begin{proposition}\label{omax2} If $\osr$ and $\ost$ are operator systems, then the following statements are equivalent for
$x\in\osr\otimes\ost$:
\begin{enumerate}
\item $x\in (\osr\omax\ost)_+$;
\item for every $\varepsilon>0$, there exist $N\in\mathbb N$, $G=[g_{ij}]_{i,j=1}^N\in \M_N( \osr)_+$, and $H=[h_{ij}]_{i,j=1}^N\in \M_N( \ost)_+$
such that
\[
x+\varepsilon(e_\osr\otimes e_\ost) =\sum_{i=1}^N\sum_{j=1}^N g_{ij}\otimes h_{ij}.
\]
\end{enumerate}
\end{proposition}

\begin{proof} Select $x\in (\osr\omax\ost)_+$, and let $\varepsilon>0$ and set $\delta=\varepsilon/2$. The elements 
$y=x+\varepsilon(e_\osr\otimes e_\ost)$ and $y-\delta(e_\osr\otimes e_\ost)$ also belong to $(\osr\omax\ost)_+$, as 
$(\osr\omax\ost)_+$ is a cone. The condition $y-\delta(e_\osr\otimes e_\ost)\in(\osr \omax\ost)_+$ implies $y$ is a strictly
positive element of $(\osr\omax\ost)_+$. Thus, by \cite[Lemma 2.7]{farenick--kavruk--paulsen--todorov2014}, 
there exist $N\in\mathbb N$ and matrices $ [g_{ij}]_{i,j=1}^N\in \M_N( \osr)_+$  and $ [h_{ij}]_{i,j=1}^N\in \M_N( \ost)_+$
such that $y=\displaystyle\sum_{i=1}^N\sum_{j=1}^N g_{ij}\otimes h_{ij}$.

Conversely, let $\varepsilon>0$. By assumption,  
there exist $N\in\mathbb N$ and matrices $ G=[g_{ij}]_{i,j=1}^N\in \M_N( \osr)_+$  and $H= [h_{ij}]_{i,j=1}^N\in \M_N( \ost)_+$
such that $x+\varepsilon(e_\osr\otimes e_\ost)=\displaystyle\sum_{i=1}^N\sum_{j=1}^N g_{ij}\otimes h_{ij}$. 
Consider the vector $\xi\in\mathbb C^N\otimes\mathbb C^N$ given by 
$\xi=\displaystyle\sum_{i=1}^N\displaystyle\sum_{j=1}^N
f_i\otimes f_j$, where $\{f_1,\dots, f_N\}$ are the canonical orthonormal basis vectors for $\mathbb C^N$. Hence, for the
linear map $\alpha:\mathbb C\rightarrow\mathbb C^N\otimes\mathbb C^N$ given by $\alpha(\zeta)=\zeta\xi$, for $\zeta\in\mathbb C$,
we have
\[
\alpha^*(G\otimes H)\alpha= \displaystyle\sum_{i=1}^N\sum_{j=1}^N g_{ij}\otimes h_{ij} = x+\varepsilon(e_\osr\otimes e_\ost).
\]
As $\varepsilon>0$ is arbitrary, we deduce $x\in(\osr\omax\ost)_+$.
\end{proof}

The following theorem is established in \cite{farenick--paulsen2012}. 

\begin{theorem}[Tensor Duality]\label{tensor duality} If $\osr$ and $\ost$ are finite-dimensional operator systems, then
\[
\left(\osr\omin\ost\right)^d = \osr^d\omax \ost^d .
\]
\end{theorem}

Turning now to tensor cones, ignoring
the matrix cones $\M_n(\osr\otimes_\sigma\ost)_+$ for all $n\ge2$, the next result creates the setting for the study herein.

\begin{proposition}\label{tc} For any finite-dimensional operator systems $\osr$ and $\ost$, and any operator system tensor product structure
$\otimes_\sigma$ on $\osr\otimes\ost$, the cone $(\osr\otimes_\sigma\ost)_+$ is a tensor cone for $\osr_+$ and $\ost_+$.
\end{proposition}

\begin{proof} 
Suppose that $x\in \osr_+\otimes_{\rm sep}\ost_+$; thus, there exist $k\in\mathbb N$, $a_j\in\osr_+$, and $b_j\in\ost_+$ such that
$x=\displaystyle\sum_{j=1}^k a_j\otimes b_j$. Let $\varepsilon>0$ and consider $\varepsilon(e_\osr\otimes e_\ost)+x$. Set 
$a_0=\varepsilon e_\osr$ and $b_0=e_\ost$ so that 
\[
\varepsilon(e_\osr\otimes e_\ost)+x=\sum_{j=0}^k a_j\otimes b_j.
\]
Let $P=\displaystyle\sum_{p=0}^k a_p\otimes e_{pp}$ and $Q=\displaystyle\sum_{q=0}^k a_q\otimes e_{qq}$, which are positive
matrices in $\M_{k+1}(\osr)$ and $\M_{k+1}(\ost)$, respectively, and let $\xi\in \mathbb C^{k+1}\otimes\mathbb C^{k+1}$ be the
vector $\xi=\displaystyle\sum_{j=0}^k f_j\otimes f_j$, where $\{f_0,\dots,f_k\}$ denotes the canonical orthonormal basis for $\mathbb C^{k+1}$.
Hence, if $\alpha:\mathbb C\rightarrow \mathbb C^{k+1}\otimes\mathbb C^{k+1}$ is the linear map $\alpha(\zeta)=\zeta\xi$, for 
$\zeta\in\mathbb C$, then 
\[
\alpha^*(P\otimes Q)\alpha= \sum_{j=0}^k a_j\otimes b_j = \varepsilon(e_\osr\otimes e_\ost)+x,
\]
which proves that $x \in (\osr\omax\ost)_+$.
Thus,
\[
\osr_+\otimes_{\rm sep}\ost_+\subseteq (\osr\omax\ost)_+ \subseteq (\osr\otimes_\sigma\ost)_+,
\]
where the second inclusion is by virtue of the inclusion sequence (\ref{e:ten rel}).

Consider now an element $x\in (\osr\omin\ost)_+$, 
and suppose that $\varphi:\osr\rightarrow\mathbb C$ and $\vartheta:\ost\rightarrow\mathbb C$
are positive linear functionals. Without loss of generality, we may assume that have been normalised to be unital.
As unital positive linear functionals are ucp maps, $(\varphi\otimes\vartheta)[x]$ is positive in $\mathbb C\otimes\mathbb C=\mathbb C$, by definition
of the minimal operator system tensor product. In other words, $x\in \osr_+\otimes_{{\rm sep}^*}\ost_+$. Thus,
\[
(\osr\otimes_\sigma\ost)_+ \subseteq (\osr\omin\ost)_+ \subseteq \osr_+\otimes_{{\rm sep}^*}\ost_+,
\]
where the first inclusion is, again, by virtue of the inclusion sequence (\ref{e:ten rel}).

Hence, $\osr_+\otimes_{\rm sep}\ost_+ \subseteq (\osr\otimes_\sigma\ost)_+ \subseteq \osr_+\otimes_{{\rm sep}^*}\ost_+$, which proves that 
$(\osr\otimes_\sigma\ost)_+$ is a tensor cone for $\osr_+$ and $\ost_+$.
\end{proof}

%%%%%%%%%%%%%%%%%%%%%%%%%%%%%%
\section{Toeplitz and Fej\'er-Riesz Operator Systems}

\subsection{Canonical linear bases}

The operator systems under study in this paper are the Toeplitz and Fej\'er-Riesz operator systems $C(S^1)^{(n)}$
and $C(S^1)_{(n)}$, for $n\geq 2$. The operator system structure on the
Toeplitz matrices arises from considering $C(S^1)^{(n)}$ as an operator subsystem of the unital C$^*$-algebra
$\M_n(\mathbb C)$, while the operator system structure on the trigonometric polynomials arises from considering
$C(S^1)_{(n)}$ as an operator subsystem of the unital abelian C$^*$-algebra $C(S^1)$ of continuous complex-valued
functions on the unit circle $S^1$. Thus, the Archimedean order unit for $C(S^1)^{(n)}$ is the identity, and for
$C(S^1)_{(n)}$ it is the constant function $z\mapsto 1$.

The functions $\chi_\ell:S^1\rightarrow\mathbb C$ given by $\chi_\ell(z)=z^\ell$, for $\ell=-n+1,\dots,n-1$,
form a linear basis for $C(S^1)_{(n)}$, with the canonical Archimedean order unit for the operator system 
$C(S^1)_{(n)}$ being the constant function $\chi_0$. It is sometimes useful to note
that the operator system $C(S^1)_{(n)}$ 
is an operator subsystem of $C(S^1)_{(m)}$, whenever $m\geq n$. Finally, if elements $t_\ell$ 
are selected from an operator system $\ost$, then the element
$x=\displaystyle\sum_{\ell=-n+1}^{n-1} \chi_\ell\otimes t_\ell$ in the 
vector space $C(S^1)_{(n)}\otimes\ost$ has a natural representation as 
a function $S^1\rightarrow\ost$ of the form $z\mapsto \displaystyle\sum_{\ell=-n+1}^{n-1} z^\ell  t_\ell$.
 
Turning to the Toeplitz operator system,
the canonical linear basis for $C(S^1)^{(n)}$ is $\{r_\ell\}_{\ell=-n+1}^{n-1}$, where
\[
r_\ell\;=\; \left\{
       \begin{array}{lcl}
           s^\ell   &:\;&      \mbox{if }\ell\geq0  \\
            (s^*)^\ell          &:\;&       \mbox{if }\ell < 0
      \end{array}
      \right\}, \mbox{ for }\ell=0,1,\dots,n-1,
\]
and $s\in \M_n(\mathbb C)$ is the shift matrix
\[
s= \left[\begin{array}{ccccc} 0&&&&  \\ 1&0&&&  \\ &1&0&&  \\ &&\ddots&\ddots& \\ &&&1&0 \end{array}\right].
\]
Thus, if elements $t_\ell$ are selected from an operator system $\ost$, then the element
$x=\displaystyle\sum_{\ell=-n+1}^{n-1} r_\ell\otimes t_\ell$ in the 
vector space $C(S^1)^{(n)}\otimes\ost$ has a natural representation
as a matrix with entries from $\ost$: namely,
\begin{equation}\label{e:bt}
x= \left[\begin{array}{cccc} t_0 & t_{-1}  & \dots & t_{-n+1} \\ t_1 & t_0 & \ddots & \vdots \\
\vdots & \ddots & \ddots & t_{-1} \\ t_{n-1}&\dots &t_1&t_0 \end{array}\right]
\end{equation}
The canonical Archimedean order unit for the operator system $C(S^1)^{(n)}$ is
the identity matrix, $r_0$.

It is common in the matrix theory literature to refer to matrices such as those in (\ref{e:bt})
as \emph{block Toeplitz matrices} when the elements $t_\ell$ are not complex numbers.

%%%%%%%%%%%%%%%%%%%%%%%%%%%%%
\subsection{Generalised circulants}

While the Toeplitz operator system is far from being closed under multiplication, it has numerous operator subsystems that are
abelian C$^*$-algebras.
For $\theta\in\mathbb R$, let $u_\theta$ be the unitary Toeplitz matrix
\[
u_\theta=r_1+e^{i\theta}r_{-n+1} = \left[\begin{array}{cccc} 0 & \dots & 0 & e^{i\theta} \\ 1 & 0 & \ddots & 0 \\
\vdots & \ddots & \ddots & \vdots \\ 0&\dots &1&0 \end{array}\right].
\]
The minimal annihilating polynomial for $u_\theta$ is $z^n-e^{i\theta}$, which has $n$ distinct roots in $S^1$, and,
for any complex polynomial $g(z)=\displaystyle\sum_{k=0}^{n-1}\alpha_kz^k$, the matrix
\[
g(u_\theta)= \left[\begin{array}{cccc} \alpha_0 & \alpha_{n-1}e^{i\theta} & \dots & \alpha_1 e^{i\theta} \\ \alpha_1 & \alpha_0 & \ddots & \vdots \\
\vdots & \ddots & \ddots & \alpha_{n-1}e^{i\theta} \\ \alpha_{n-1}&\dots &\alpha_1&\alpha_0 \end{array}\right]
\]
is normal and Toeplitz. Moreover, the operator systems
\[
C^{n,\theta}=\left\{g(u_\theta)\,|\, g\in\mathbb C[z]\right\},
\]
for $\theta\in\mathbb R$, are unital abelian C$^*$-subalgebras of $C(S^1)^{(n)}$, yielding 
the algebra of circulant matrices, when $\theta=0$,
and the algebra of skew-circulant matrices, when $\theta=\pi$. Elements of $C^{n,\theta}$ are called \emph{generalised circulants}.
Note that, by the Spectral Theorem, the positive cone of $C^{n,\theta}$ is affinely homeomorphic to the nonnegative orthant of 
$\mathbb R^n$; hence, $\left(C^{n,\theta}\right)_+$ is a simplicial cone. 
By diagonalising matrices in $C^{n,\theta}$ via a single unitary
matrix $U_\theta$, for each $\theta\in\mathbb R$, we have the following isomorphisms in the category $\mathfrak S_1$:
\begin{equation}\label{e:circ}
C^{n,\theta}  \simeq \cstar(\mathbb Z_n),
\end{equation}
for all $\theta\in\mathbb R$, where $\cstar(\mathbb Z_n)$ is the group C$^*$-algebra 
of the finite abelian group $\mathbb Z_n$.  

As noted in \cite[\S5]{connes-vansuijlekom2021}, the operator systems $C^{n,\theta}$ are self-dual and the Toeplitz operator system
$C(S^1)^{(n)}$ is the (completely positive) truncation of the circulant operator system $C^{2n-1,\theta} $, for $\theta=0$, to the $n\times n$ upper-left corner. Unfortunately,
this completely positive map does not send the positive cone of $C^{2n-1,\theta} $, for $\theta=0$, onto the positive cone of $C(S^1)^{(n)}$.
For example, the positive Toeplitz matrix $\left[\begin{array}{cc} 1 & \zeta^{-1} \\ \zeta & 1 \end{array}\right]$, for $z \in S^1$, is the upper $2\times 2$ corner of a 
$3\times 3$ positive circulant matrix if and only if $\zeta$ is a cube root of unity.

%%%%%%%%%%%%%%%%%%%%%%%%%%%%
\subsection{The matrices $R_n$ and $T_n$}

\begin{definition} 
$R_n\in C(S^1)^{(n)}\otimes C(S^1)^{(n)}$ and
$T_n\in C(S^1)^{(n)}\otimes C(S^1)_{(n)}$ are the matrices given by
 \[
R_n =\sum_{\ell=-n+1}^{n-1} r_\ell\otimes r_{\ell}
\,\mbox{ and }\,
T_n =\sum_{\ell=-n+1}^{n-1} r_\ell\otimes \chi_{\ell}.
\]
The matrix $T_n$ is called the \emph{universal positive $n\times n$ Toeplitz matrix}.
\end{definition}

It is evident $T_n$ is positive. To show $R_n$ is positive, one need only note that
$R_n$ coincides with the matrix designated by (*) on page 36 of \cite{Paulsen-book}, and it is shown 
in \cite{Paulsen-book} that the matrix (*) is positive. 

In addition to being positive, the matrix $R_n$ separable (see Proposition \ref{R_n is separable}), while $T_n$ entangled \cite[Corollary 7.7]{farenick2021}.

An important result of Choi \cite{choi1972} shows that the complete positivity of a linear map $\phi$ of the matrix algebra $\M_n(\mathbb C)$ to an operator system $\ost$
can be confirmed by determining whether a single matrix, namely the \emph{Choi matrix} $C_\phi=\displaystyle\sum_{i,j} e_{ij}\otimes\phi(e_{ij})$, is positive in 
$\M_n(\mathbb C)\omin \ost$.
The following theorem shows how the matrix $T_n$ functions in a similar way for linear maps on $C(S^1)_{(n)}$.

\begin{proposition}\label{pos-cp} Let $\ost$ be an operator system.
\begin{enumerate}
\item Every positive linear map $\phi:C(S^1)^{(n)}\rightarrow\ost$ is completely positive.
\item The following statements are equivalent for a linear map $\phi:C(S^1)_{(n)}\rightarrow\ost$:
    \begin{enumerate}
    \item the matrix $\displaystyle\sum_{\ell=-n+1}^{n-1}r_\ell\otimes\phi(\chi_\ell)$ is positive in $C(S^1)^{(n)}\omin C(S^1)_{(n)}$;
    \item $\phi$ is completely positive.
    \end{enumerate}
\end{enumerate}
\end{proposition}

\begin{proof}
The assertion that positive linear maps $\phi:C(S^1)^{(n)}\rightarrow\ost$ are completely positive is established in \cite[Lemma 2.5]{farenick2021}.

For the second assertion, because 
\[
C(S^1)^{(n)}\omin C(S^1)_{(n)} \subseteq_{\rm coi} \M_n(\mathbb C)\omin C(S^1)_{(n)},
\]
it is clear, by definition of complete positivity, that  $\displaystyle\sum_{\ell=-n+1}^{n-1}r_\ell\otimes\phi(\chi_\ell)$ is positive in $C(S^1)^{(n)}\omin C(S^1)_{(n)}$,
if $\phi$ is completely positive.

Conversely, suppose $\phi:C(S^1)^{(n)}\rightarrow\ost$ is a linear map for which the matrix 
$\displaystyle\sum_{\ell=-n+1}^{n-1}r_\ell\otimes\phi(\chi_\ell)$ is positive in $C(S^1)^{(n)}\omin \ost$. Again, using the fact that
\[
C(S^1)^{(n)}\omin \ost \subseteq_{\rm coi} \M_n(\mathbb C)\omin \cstare(\ost),
\]
we may consider $\phi$ to be a linear map $\phi:C(S^1)^{(n)}\rightarrow\cstare(\ost)$ for which the matrix
$\displaystyle\sum_{\ell=-n+1}^{n-1}r_\ell\otimes\phi(\chi_\ell)$ is positive in $C(S^1)^{(n)}\omin \cstare(\ost)$, Thus, by
the universal property of $T_n$ \cite[Theorem 7.7]{farenick2021}, there is a completely positive linear map
$\psi:C(S^1)_{(n)}\rightarrow\cstare(\ost)$ such that $\psi(\chi_\ell)=\phi(\chi_\ell)$, for every $\ell$. Hence, because
$\psi$ and $\phi$ agree on a linear basis for $C(S^1)_{(n)}$, they must be equal, which implies that $\phi$
is completely positive.
\end{proof}

\begin{corollary} If $\phi:C(S^1)_{(n)}\rightarrow\ost$ is a positive, but not completely positive, linear map, then the matrix
$\displaystyle\sum_{\ell=-n+1}^{n-1}r_\ell\otimes\phi(\chi_\ell)$ is nonpositive in $C(S^1)^{(n)}\omin C(S^1)_{(n)}$.
\end{corollary}

Proposition \ref{pos-cp} may be used to show certain positive linear maps on $C(S^1)_{(n)}$ are not
completely positive.
The following example is inspired by \cite[Appendix A2]{arveson1969} (see also \cite[Example 2.2]{Paulsen-book}).

\begin{example} The linear map $\phi:C(S^1)_{(2)}\rightarrow C(S^1)^{(2)}$ given by
\[
\phi(\alpha_{-1}z^{-1} +\alpha_0 + \alpha_1 z^1)=\left[ \begin{array}{cc} \alpha_0 & 2\alpha_{-1} \\ 2\alpha_1 & \alpha_0\end{array}\right]
\]
is positive but not completely positive.
\end{example}

\begin{proof} Because $\overline{\alpha}_{1}z^{-1} +\alpha_0 + \alpha_1 z^1$ is positive if and only if $\alpha_0+2\Re(\alpha_1z)\geq 0$ for all $z\in S^1$,
$\alpha_0$ and $\alpha_1$ must necessarily satisfy $\alpha_0\geq 2|\alpha_1|$, which in turn implies that 
the matrix $\left[ \begin{array}{cc} \alpha_0 & 2\overline{\alpha}_{1} \\ 2\alpha_1 & \alpha_0\end{array}\right]$
is positive. Hence, the linear map $\phi$ is positive. 

Consider the matrix
\[
x= \sum_{\ell=-1}^1 r_\ell \otimes \phi(\chi_\ell) = \left[ \begin{array}{cc} r_0 & 2r_{-1} \\ 2r_1 & r_0\end{array}\right].
\]
If $\xi\in\mathbb C^4=\mathbb C^2\otimes\mathbb C^2$ is $\xi=e_1-e_4$, then $\langle x\xi,\xi\rangle=-2<0$, which shows that
$x$ is not positive; hence, $\phi$ is not completely positive, by Proposition \ref{pos-cp}.
\end{proof} 

On the other hand:

\begin{example} The linear map $\phi:C(S^1)_{(n)}\rightarrow C(S^1)^{(n)}$ defined on the basis elements
by $\phi(\chi_\ell)=r_\ell$ is completely positive.
\end{example}

\begin{proof} The matrix $\displaystyle\sum_{\ell=-n+1}^{n-1} r_\ell\otimes\phi(\chi_\ell)$ is $R_n$, which is positive. Therefore,
by Proposition \ref{pos-cp}, $\phi$ is completely positive.
\end{proof}

%%%%%%%%%%%%%%%%%%%%%%%%
\subsection{From positive matrices to completely positive linear maps}

For any operator system $\ost$ and Toeplitz matrix $x=\sum_\ell r_\ell\otimes t_\ell\in C(S^1)^{(n)}\otimes \ost$, a linear map
$\hat x: C(S^1)_{(n)}\rightarrow \ost$ is induced in which the action of $\hat x$ is given by
\begin{equation}\label{e:hat}
\hat x(f)=\sum_{\ell=-n+1}^{n-1}\hat f(-\ell) t_\ell,\mbox{ for all }f\in C(S^1)_{(n)}.
\end{equation}
Moreover, the matrix $x$ is positive in $C(S^1)^{(n)}\omin \ost$ if and only if the linear map $\hat x: C(S^1)_{(n)}\rightarrow \ost$ 
is completely positive. 

In the case where the matrix in question is $R_n$, we obtain the map
\[
\hat R_n (f) = \sum_{\ell=-n+1}^{n-1} f(-\ell)r_\ell
=
\left[\begin{array}{cccc} \hat f(0) & \hat f({-1}) & \dots & \hat f({-n+1}) \\ \hat f(1) & \hat f(0) & \ddots & \vdots \\
\vdots & \ddots & \ddots & \hat f({-1}) \\ \hat f({n-1})&\dots &\hat f(1)&\hat f(0) \end{array}\right],
\]
which is unitarily equivalent to the Fourier matrix for $f\in C(S^1)_{(n)}$.

In the case where $x=T_n(\lambda)\otimes T_n(\mu)$ for some $\lambda,\mu\in S^1$, expressing $x$ as 
\[
x=\sum_{\ell=-n+1}^{n-1} r_\ell \otimes \lambda^{\ell}T_n(\mu)
\]
leads to 
\[
\hat x(f) = \left(\sum_{\ell=-n+1}^{n-1} \hat f(-\ell)\lambda^\ell\right) T_n(\mu),
\]
in which case the positivity of $f$ implies the positivity of the scalar that multiplies the positive rank-1 matrix $T_n(\mu)$.
Hence, the following proposition is proved.

\begin{proposition}\label{induced positive linear functionals}
For each $\lambda\in S^1$, the linear functional $\varphi_\lambda:C(S^1)_{(n)}\rightarrow\mathbb C$, defined by 
\[
\varphi_\lambda(f) = \sum_{\ell=-n+1}^{n-1} \hat f(-\ell)\lambda^\ell,
\]
is positive.
\end{proposition}

%%%%%%%%%%%%%%%%%%%%%%%
\subsection{Linear $*$-preserving idempotents}

The following observation situates the Toeplitz operator system within the larger matrix algebra.

\begin{proposition}\label{idempotent} There exists a linear map
$\mathcal E_n:\M_n(\mathbb C)\rightarrow\M_n(\mathbb C)$, for every $n\geq 2$, 
such that
\begin{enumerate}
\item $\mathcal E_n(x^*)=\mathcal E_n(x)^*$, for all $x\in\mathcal M_n(\mathbb C)$,
\item $\mathcal E_n$ is idempotent, and
\item the range of $\mathcal E_n$ is $C(S^1)^{(n)}$.
\end{enumerate}
Furthermore, $\mathcal E_2$ is positive, but not completely positive.
\end{proposition}

\begin{proof} Let $\mathcal E_n$ be the map that 
averages the entries along each of the super diagonals of a matrix $x$ and replaces each entry 
along the super diagonal with that average. Thus, $\mathcal E_n$ is linear, satisfies 
$\mathcal E_n(x^*)=\mathcal E_n(x)^*$, for all $x\in\mathcal M_n(\mathbb C)$, and maps into
$C(S^1)^{(n)}$. 
If $x$ is a Toeplitz matrix, then the super diagonals are constant, leaving an average of the entries unchanged. Hence,
$\mathcal E_n$ maps onto $C(S^1)^{(n)}$ and
$\mathcal E_n^2=\mathcal E_n$.

To show that $\mathcal E_2$ is positive, 
let $x=\left[\begin{array}{cc}\alpha & \overline{\gamma} \\ \gamma & \beta\end{array}\right]$ be a positive
$2\times 2$ matrix. Thus, $\alpha\geq 0$, $\beta \geq 0$, and $|\gamma|^2\leq \alpha\beta$.
If 
\[
y=\mathcal E_2(x) = \left[ \begin{array}{cc} \frac{\alpha+\beta}{2} & \overline{\gamma} \\ 
\gamma & \frac{\alpha+\beta}{2} \end{array}\right],
\]
then the diagonal entries of $y$ are nonnegative and 
\[
|\gamma| \leq \sqrt{\alpha\beta} \leq \frac{\alpha+\beta}{2},
\]
where the second inequality is the arithmetic-geometric mean inequality. Thus, the matrix $y$ is
positive, proving $\mathcal E_2$ is a positive linear map.

The Choi matrix $\left[ \mathcal E_2(e_{ij})\right]_{i,j=1}^2$
for $\mathcal E_2$ is easily seen to be nonpositive; hence, $\mathcal E_2$ is not completely
positive.
\end{proof}

\begin{corollary}[Selfadjoint Toeplitz matrices]\label{saT} For every $n\geq 2$ and complex $*$-vector space $\mathcal V$,
\[
C(S^1)^{(n)}_{\rm sa}\otimes_{\mathbb R} \mathcal V_{\rm sa} 
=
\left( C(S^1)^{(n)}  \otimes_{\mathbb C} \mathcal V  \right)_{\rm sa}.
\]
\end{corollary}

\begin{proof} If $x=\displaystyle\sum_{r=-n+1}^{n-1} r_\ell\otimes v_\ell \in 
\left( C(S^1)^{(n)}  \otimes_{\mathbb C} \mathcal V  \right)_{\rm sa}$, then 
$x\in \left( \M_n(\mathbb C) \otimes_{\mathbb C} \mathcal V  \right)_{\rm sa}$ as well. Therefore, by
\cite[Lemma 3.7]{paulsen--todorov--tomforde2010}, there are selfadjoint $h_k\in \M_n(\mathbb C)$
and $s_k\in\mathcal V_{\rm sa}$ such that
\[
x=\sum_{r=-n+1}^{n-1} r_\ell\otimes v_\ell =\sum_{j=1}^m h_k\otimes s_k.
\]
Apply the ampliation $\mathcal E_n\otimes\mbox{\rm Id}_{\mathcal V}$ to obtain
\[
(\mathcal E_n\otimes\mbox{\rm Id}_{\mathcal V})[x]
=\sum_{j=1}^m \mathcal E_n(h_k)\otimes s_k= 
\sum_{r=-n+1}^{n-1} \mathcal E_n(r_\ell)\otimes v_\ell
=\sum_{r=-n+1}^{n-1}  r_\ell \otimes v_\ell=x.
\]
As each $\mathcal E_n(h_k)$ is a selfadjoint Toeplitz matrix, the proof is complete.
\end{proof}

As a consequence of Corollary \ref{no-wep1}, if $n\geq 2$ and 
$\mathcal G_n:\M_n(\mathbb C)\rightarrow\M_n(\mathbb C)$ is an idempotent linear map
with range $C(S^1)^{(n)}$, then $\mathcal G_n$ can not be completely positive.
The situation, however, is rather different for the circulant operator systems $C^{n,\theta}$.

\begin{proposition}\label{idempotent2} For every $n\geq 2$ and $\theta\in \mathbb R$, there exists a linear map
$\mathcal F_{n,\theta}:\M_n(\mathbb C)\rightarrow\M_n(\mathbb C)$  
such that
\begin{enumerate}
\item $\mathcal F_{n,\theta}$ is completely positive,
\item $\mathcal F_{n,\theta}$ is idempotent, and
\item the range of $\mathcal F_{n,\theta}$ is $C^{n,\theta}$.
\end{enumerate}
\end{proposition}

\begin{proof} 
Fix $n\geq 2$ and $\theta\in \mathbb R$, and let $(\lambda_1,\dots,\lambda_n)$ be a fixed ordering of the (necessarily distinct)
eigenvalues of the unitary matrix $u_\theta$.
By the Spectral Theorem, there is a unitary $v_\theta$ such that
$v_\theta^* f(u_\theta) v_\theta$ is a diagonal matrix, for every polynomial $f\in\mathbb C[z]$. Thus, the
automorphism $\vartheta_n(x)=  v_\theta^* x v_\theta$ of $\M_n(\mathbb C)$ maps $C^{n,\theta}$ into the algebra of diagonal matrices. In fact this 
map is onto, because, for any $n$-tuple $(w_1,\dots, w_n)$ of complex numbers, there is an interpolating polynomial $g\in\mathbb C[z]$
that sends each $\lambda_j$ to $w_j$; that is, the diagonal matrix determined by $(w_1,\dots, w_n)$ is $v_\theta^* g(u_\theta) v_\theta$.
The linear map $E_n$ that sends every $n\times n$ matrix to its diagonal is a completely positive 
idempotent that maps $\M_n(\mathbb C)$ onto the subalgebra of diagonal matrices. Thus,  define $\mathcal F_{n,\theta}$
to be $\vartheta_n^{-1}\circ E_n$.
\end{proof}

%%%%%%%%%%%%%%%%%%%%%%%%%%%%%%%%%%%
\subsection{Geometry of the positive cones}

This section concludes with observations
regarding the geometric nature of the positive cones of the operator systems discussed herein.

Evaluating a universal Toeplitz matrix $T_n$ at a point $\lambda\in S^1$ yields a positive Toeplitz matrix
$T_n(\lambda)$ in which $n^{-1}T_n(\lambda)$ is a rank-1 projection.
Such matrices comprise the extremal rays of the cone $\left(C(S^1)^{(n)}\right)_+$.

\begin{proposition}[Extremal Rays]\label{cvS}\cite[Propositions 4.5 
and 4.8]{connes-vansuijlekom2021} The extremal rays of the cone 
$\left(C(S^1)^{(n)}\right)_+$ are positive scalar multiples of matrices of the form $T_n(\lambda)$, for 
$\lambda\in S^1$. The extremal rays of the cone 
$\left(C(S^1)_{(n)}\right)_+$ are positive scalar multiples of those elements
 $f\in \left(C(S^1)_{(n)}\right)_+$ with the property that
  $f(\omega)=0$, for some $\omega\in\mathbb C\setminus\{0\}$, only if $\omega\in S^1$.
\end{proposition}

\begin{proposition}\label{not classical} If $n\geq 2$, 
the cone $\left(C^{n,\theta}\right)_+$ is simplicial, for every $\theta\in\mathbb R$, whereas the cones
$\left(C(S^1)^{(n)}\right)_+$ and
 $\left(C(S^1)_{(n)}\right)_+$ are not simplicial.
\end{proposition}

\begin{proof} The discussion leading up to the statement of the proposition establishes that 
$\left(C^{n,\theta}\right)_+$ is a simplicial cone. 

By Proposition \ref{cvS}, 
the extremal rays of $\left(C(S^1)^{(n)}\right)_+$ are generated by $T_n(\lambda)$ for $\lambda\in S^1$,
and so there cannot be a linear isomorphism that maps $\left(C(S^1)^{(n)}\right)_+$ 
onto any simplicial cone, as a simplicial cone
has only finitely many extremal rays. 

Likewise, the extremal rays of $\left(C(S^1)_{(n)}\right)_+$ 
are generated by Laurent polynomials that are nonnegative on $S^1$ and 
whose zeros are contained in $S^1$. As 
\[
f_\lambda(z)=\lambda z^{-1}+2+\lambda^{-1} z
\]
is one such Laurent polynomial, for each $\lambda\in S^1$,
the cone $\left(C(S^1)^{(n)}\right)_+$ has infinitely many extremal rays.
\end{proof}

\begin{corollary}\label{nope} If $\osr_+$ and $\ost_+$ are Toeplitz or Fej\'er-Riesz cones, then
\[
\osr_+\otimes_{\rm sep}\ost_+\not= \osr_+\otimes_{{\rm sep}^*} \ost_+.
\]
\end{corollary}

\begin{proof} If equality held, then the main result of \cite{aubrun--ludovico--palazuelos2021} 
implies at least one of the Toeplitz cones $\osr_+$ or $\ost_+$
would be simplicial; however, such an implication is
impossible by Proposition \ref{not classical}.
\end{proof}

Proposition \ref{cvS} provides a method via convexity to generate all $n\times n$ positive
Toeplitz matrices, foreshadowing a similar result (Proposition \ref{nc cnvx})
for positive block Toeplitz matrices.

\begin{definition} For $n\ge 2$, let $\mathbb T^n$ denote the $n$-torus $S^1\times \cdots \times S^1$ and let 
$\mathfrak P_n\subset\mathbb T^n$ be the set of all $n$-tuples of the form
$(\lambda,\lambda^2,\dots,\lambda^n)$, for $\lambda\in S^1$. 
\end{definition}

\begin{proposition}\label{foreshadow} $\xi=(\xi_1,\dots,\xi_{n-1})\in\mbox{\rm Conv}\,\mathfrak P_{n-1}$ if and only if
the Toeplitz matrix 
\[
x(\xi)= \left[\begin{array}{cccc} 1 & \overline\xi_1 & \dots & \overline\xi_{n-1} \\ \xi_1 & 1 & \ddots & \vdots \\
\vdots & \ddots & \ddots & \overline\xi_1 \\ \xi_{n-1}&\dots &\xi_1&1 \end{array}\right]
\]
is positive.
\end{proposition}

\begin{proof} If $x(\xi)$ is positive, then it is a convex combination $\displaystyle\sum_{j=1}^k\alpha_j T_n(\lambda_j)$
of matrices of the form $T_n(\lambda_j)$, where
$\lambda_j\in S^1$. Thus, $\xi_\ell=\displaystyle\sum_{j=1}^k\alpha_j\lambda_j^\ell$, for each $\ell$.

Conversely, if $\xi\in \mbox{\rm Conv}\,\mathfrak P_{n-1}$, then there exist convex coefficients $\alpha_j$ and 
elements $\lambda_j\in S^1$ such that each
$\xi_\ell=\displaystyle\sum_{j=1}^k\alpha_j\lambda_j^\ell$. Because 
$\displaystyle\sum_{j=1}^k\alpha_j T_n(\lambda_j)=x(\xi)$, the matrix $x(\xi)$ is positive.
\end{proof}

%%%%%%%%%%%%%%%%%%%%%%%%%%%%%%
\section{Separability}

%%%%%%%%%%%%%%%%%%%%%%%%%%%%%%
\subsection{The Gurvits Separation Theorem}

An elegant theorem of Gurvits, explained in a joint paper with Burnam \cite{gurvits--burnam2002}, 
states 
\[
\left(C(S^1)^{(n)}\omin \M_m(\mathbb C)\right)_+ = (C(S^1)^{(n)})_+\otimes_{\rm sep} \left(\M_m(\mathbb C)\right)_+.
\]
Given the non-simplicial nature of the cones $(C(S^1)^{(n)})_+$ and $\left(\M_m(\mathbb C)\right)_+$, Gurvits' result
is remarkable, and a great deal of effort has gone into understanding this separability result. For this reason a number of 
expositions of Gurvits' original proof, 
such as those in \cite{qian--chen--chu2020,yang--xie--stoica2016}, or the development of alternative proofs, 
such as those in \cite{ando2004,ando2013,farenick--mcburney2023}, 
have appeared in the literature.

Although all the steps in Gurvits' argument are indicated in the joint paper with Burnam \cite{gurvits--burnam2002}, 
the proof of one of the key steps, which is stated below as Lemma \ref{douglas lem2}, 
is missing in every exposition I am aware of. Therefore, the purpose of this expository
subsection is to fill this gap in the published literature, 
drawing inspiration from the well-known majorisation theorem of Douglas \cite{douglas1966} 
to prove the aforementioned ``missing'' lemma of \cite{gurvits--burnam2002}. For completeness, 
the entirety of Gurvits' argument is also presented.

\begin{lemma}\label{douglas lem2} The following statements are equivalent for $q\times r$
complex matrices $x$ and $y$:
\begin{enumerate} 
\item $yy^*=xx^*$;
\item there is a $r\times r$ unitary
matrix $w$ such that $y=xw$.
\end{enumerate}
\end{lemma}

\begin{proof} It is clear that (2) implies (1). To prove that (1) implies (2), let
$\mbox{ran}\,a$ and $\ker a$ denote, respectively, the range and null-space of a linear 
transformation $a$, and view the matrices
$x$ and $y$ as linear transformations of $\mathbb C^q$ into $\mathbb C^r$, 
which we consider as Hilbert spaces with respect to their standard
inner products.
The equation $xx^*=yy^*$ implies that the function $y^*\xi\mapsto x^*\xi$ is a well-defined isometric linear
transformation $\tilde w_0$ of $\mbox{ran}\,y^*$ into $\mbox{ran}\,x^*$. Because $\tilde w_0$ has 
a linear inverse given by $x^*\xi\mapsto y^*\xi$, the linear transformation $\tilde w_0$ is an
isomorphism of the vector spaces $\mbox{ran}\,y^*$ and $\mbox{ran}\,x^*$. Extend the domain of $\tilde w_0$
to all of $\mathbb C^q$ to obtain a linear transformation $w_0:\mathbb C^q\rightarrow\mathbb C^r$ whereby
$w_0\gamma=\tilde w_0\gamma$,  for $\gamma\in\mbox{ran}\,y^*$, and $w_0\eta=0$, 
for $\eta\in\left(\mbox{ran}\,y^*\right)^\bot=\ker y$.

An arbitrary vector $\omega\in\mathbb C^q$ has the form $\omega=\gamma+\eta$, for some 
$\gamma  \in\mbox{ran}\,y^*$ and $\eta\in \ker y$. Thus, if $\gamma=y^*\xi \in\mbox{ran}\,y^*$,
for some $\xi\in\mathbb C^r$,
then 
\[
y\omega=
y\left(y^*\xi+\eta\right)= yy^*\xi = xx^*\xi = x\left(w_0(y^*\xi)+w_0\eta\right) = xw_0(y^*\xi+\eta)=xw_0\omega.
\]
Hence, $y=xw_0\omega$, for all vectors $\omega\in\mathbb C^q$, proving that $y=xw_0$.

The isomorphism of the subspaces $\mbox{ran}\,y^*$ and $\mbox{ran}\,x^*$ of $\mathbb C^q$ implies an
isomorphism of their orthogonal complements, $\ker y$ and $\ker x$, respectively. Choose any orthonormal bases of
$\ker y$ and $\ker x$, respectively, and let $\tilde w_1$ be the linear isomorphism that maps the orthonormal basis of 
$\ker y$ onto the orthonormal basis of $\ker x$, and extend the definition of $\tilde w_1$ to a linear transformation 
$w_1$ on 
all of $\mathbb C^q$ by setting $w_1\delta=0$ for all $\delta\in\left(\ker y\right)^\bot=\mbox{ran}\,y^*$.

The linear transformations $w_0$ and $w_1$ have orthogonal ranges, and $w_0^*w_0$ is the projection with
range $\mbox{ran}\,y^*$ and $w_1^*w_1$ is the (complementary) projection with range $\mbox{ran}\,x^*$. Hence,
if $w=w_0+w_1$, then $w^*w$ is the identity map on $\mathbb C^r$, which implies $w^{-1}=w^*$, proving $w$ is unitary.
Finally, the equality $\mbox{ran}\,w_1=\ker x$ implies $xw=xw_0+xw_1=xw_0=y$.
\end{proof}

Before proving Gurvits' Theorem, note that, 
if $u$ is a unitary element of a unital $C^*$-algebra $\A$, then
by functional calculus, we may evaluate the universal Toeplitz matrix $T_n$ at $u$ to obtain
\[
T_n(u)=\left[\begin{array}{cccc} 1 & u^{-1}  & \dots & u^{-n+1} \\ u & 1 & \ddots & \vdots \\
\vdots & \ddots & \ddots & u^{-1} \\ u^{n-1}&\dots &u&1 \end{array}\right],
\]
which is a positive element of $\M_n(\A)$.  

\begin{theorem}[Gurvits]\label{sep2} For all $n,m\geq 2$,
\[
(C(S^1)^{(n)})_+\otimes_{\rm sep} \left(\M_m(\mathbb C)\right)_+ =
\left(C(S^1)^{(n)}\omin \M_m(\mathbb C)\right)_+.
\]
\end{theorem}

\begin{proof} Suppose $x=\displaystyle\sum_{\ell=-n+1}^{n-1} r_\ell\otimes a_\ell
\in\left(C(S^1)^{(n)}\omin \M_m(\mathbb C)\right)_+$.
Set $q=nm$ and $r=\mbox{rank}\,x$. By 
the Spectral Theorem, 
$x$ has the form $x=udu^*$ for some unitary matrix $u$ and diagonal matrix $d$ whose diagonal entries are eigenvalues $\alpha_j$ of $x$. 
Without loss of generality, assume the first $r$ diagonal entries of $d$ are the nonzero eigenvalues of $x$ and 
let $g$ be the $r\times r$ diagonal matrix with diagonal 
entries $\alpha_j^\frac{1}{2} $
 and $z$ be the $q\times r$ rectangular matrix
$z=\begin{bmatrix}
g\\ 0 \\
\end{bmatrix}$,
where ``$0$'' in the matrix above is a $(q-r)\times q$ matrix of zeroes if $r<q$, or is absent if $r=q$.
Because 
$zz^*=d$, 
if we let $y=uz$, which is a $q\times r$ matrix,
we then obtain $yy^*=uzz^*u^*=udu^*=x$.

As below, partition the matrix $y$
as a column of $n$ matrices $y_k$, each of dimension $m\times r$, and then define $y_U$ and $y_L$:
\[
y=
\begin{bmatrix}
y_0 \\ y_1 \\ \vdots \\ y_{n-1} \\
\end{bmatrix}, \quad
y_U=
\begin{bmatrix}
y_0 \\ \vdots \\ y_{n-2} \\
\end{bmatrix},
\quad
y_L=
\begin{bmatrix}
y_1 \\ \vdots \\ y_{n-1} \\
\end{bmatrix}.
\]
The matrix $y_Uy_U^*=y_Ly_L^*$ coincides with 
the upper-left $(n-1)\times (n-1)$ block of the matrix $x$. 

Apply Lemma \ref{douglas lem2} to obtain
$y_L=y_Uw$ for some $r\times r$ unitary matrix $w$. Thus, $y_k=y_{k-1}w$, for each $k=1,\dots n-1$.
The Toeplitz structure of $x$ and the equality $y_Uy_U^*=y_Ly_L^*$ yield:
\[
a_0=y_0y_0^*,\quad
a_1= y_0wy_0^*, \quad
a_2= y_0w^2y_0^* , \quad
\dots\quad
a_{n-1}=y_0w^{n-1}y_0^* .
\]
Because $a_{-\ell}=a_\ell^*$ and $w^{-\ell}=(w^{\ell})^{*}$ for every $\ell$, we deduce 
$a_{\ell}=y_0w^{\ell}y_0^*$ for all $\ell=-n+1,...,n-1$. 

Set $z=y_0$ and  express $x$ as
\[
x=(1_n\otimes z)\left(\displaystyle\sum_{\ell=-n+1}^{n-1} r_\ell\otimes w^\ell \right)(1_n\otimes z)^*= 
\displaystyle\sum_{\ell=-n+1}^{n-1} r_\ell\otimes (zw^\ell z^*).
\]
Let $\lambda_1,\dots,\lambda_k$ denote the distinct eigenvalues of $W$.
By the Spectral Theorem, there are pairwise-orthogonal projections $p_1,\dots,p_k$ 
such that $w^\ell=\displaystyle\sum_{j=1}^k\lambda_j^\ell p_j$,
for all $\ell\in\mathbb Z$. Therefore,
\[
T_n(w)=\displaystyle\sum_{\ell=-n+1}^{n-1}r_\ell\otimes (\displaystyle\sum_{j=1}^{k} \lambda_j^\ell p_j) 
= \displaystyle\sum_{j=1}^{k}\displaystyle\sum_{\ell=-n+1}^{n-1} \lambda_j^\ell r_\ell\otimes p_j 
= \displaystyle\sum_{j=1}^{k} T_n(\lambda_j)\otimes p_j.
\]
Hence,
\[
x=(1_n\otimes z) T_n(w) (1_n\otimes z)^*=\displaystyle\sum_{j=1}^{k} T_n(\lambda_j)\otimes (zp_j z^*)
=\displaystyle\sum_{j=1}^k T_n(\lambda_j)\otimes b_j,
\]
where $b_j=zp_jz^*$ for each $j$. Thus, $x\in (C(S^1)^{(n)})_+\otimes_{\rm sep} \left(\M_m(\mathbb C)\right)_+$.
\end{proof}

It is sometimes convenient to recast Theorem \ref{sep2} as follows.

\begin{corollary}\label{g-sep} If $x=\displaystyle\sum_{\ell=-n+1}^{n-1} r_\ell\otimes a_\ell
\in\left(C(S^1)^{(n)}\omin \M_m(\mathbb C)\right)_+$, 
then there exist $k\in\mathbb N$, $\lambda_j\in S^1$, and 
$b_j\in\M_m(\mathbb C)_+$
such that  
\begin{equation}\label{e:sep3}
a_\ell= \sum_{j=1}^k \lambda_j^\ell b_j,
\end{equation}
for each $\ell=-n+1, \dots, n-1$.
\end{corollary}

%%%%%%%%%%%%%%%%%%%%
\subsection{The matrix convex hull of $\mathfrak P_n$}

Recall that $\mathfrak P_n$ is the subset of the $n$-torus consisting of $n$-tuples of the
form $(\lambda,\lambda^2,\dots,\lambda^n)$, for $\lambda\in S^1$.

One can use Gurvits' Theorem to reframe Proposition \ref{foreshadow} in the setting of matrix convexity
\cite{effros--winkler1997}.

\begin{definition} The \emph{matrix convex hull}
of $\mathfrak P_n$ is the sequence
$\mbox{\rm m-Conv}\mathfrak P_n=\left( \mathcal L_k\right)_{k\in\mathbb N}$ of subsets $ \mathcal L_k$
of $\M_k(\mathbb C)^n$ defined by
\[
\mathcal L_k=\left\{ \left(\sum_{j=1}^g\lambda_jq_j, %\sum_{j=1}^g\lambda^2_jq_j,
\dots, \sum_{j=1}^g\lambda^n_jq_j\right)\,|\,
g\in\mathbb N,\,
\lambda_j\in S^1,\, q_j\in\M_k(\mathbb C)_+,\, \sum_{j=1}^g q_j=1_k\right\}.
\]
\end{definition}

Similar to Proposition \ref{foreshadow}, we have the following result.

\begin{proposition}\label{nc cnvx} A tuple
$a=(a_1,\dots,a_{n-1})$ of $k\times k$ matrices
is in the matrix convex hull of $\mathfrak P_{n-1}$ if and only if
the block Toeplitz matrix
\[
x(a)= \left[\begin{array}{cccc} 1_k & a_1^* & \dots & a_{n-1}^* \\ a_1 & 1_k & \ddots & \vdots \\
\vdots & \ddots & \ddots & a_1^* \\  a_{n-1}&\dots &a_1&1_k\end{array}\right]
\]
is positive.
\end{proposition}

\begin{proof} The argument is the same as that of Proposition \ref{foreshadow}, except 
one uses the matrix convex combinations of equation (\ref{e:sep3}).
\end{proof} 

 %%%%%%%%%%%%%%%%%%%%%%
\subsection{Positive $2\times 2$ block Toeplitz matrices with Toeplitz blocks}

\begin{proposition}\label{2-sep}
For every $n\geq2$,
\[
C(S^1)^{(2)}{}_+ \otimes_{\rm sep} C(S^1)^{(n)}_+ = \left(C(S^1)^{(2)}\omin C(S^1)^{(n)}\right)_+.
\]
\end{proposition}

\begin{proof} By the Gurvits Separation Theorem, there exist $c_j\in\M_2(\mathbb C)_+$ and 
$s_j\in C(S^1)^{(n)}_+$ such that
\[
x=\sum_{j=1}^{m} c_j\otimes s_j.
\]
The idempotent 
$\mathcal E_2:\M_2(\mathbb C)\rightarrow C(S^1)^{(2)}$ 
(Proposition \ref{idempotent})) is a positive linear map, 
and so
\[
x=\mathcal E_2^{(n)}(x)= \sum_{j=1}^{m} \mathcal E_2(c_j)\otimes s_j,
\]
which implies the Toeplitz matrices $b_j=\mathcal E_2(c_j)$ are positive.
\end{proof}

It was noted in \cite{farenick--mcburney2023} that, if 
\[
x=\left[ \begin{array}{cc} a & c^* \\ c & a \end{array}\right]
\]  
is a positive invertible block-Toeplitz matrix, and if both $a$ and $c$ are $2\times 2$
Toeplitz matrices, then $x$ is separable in the operator system
$C(S^1)^{(2)}\omin C(S^1)^{(2)}$. 
The methods of the present paper yield the following improved result, 
in which invertibility and $2\times 2$ $a$ and $c$ are not required,
as a consequence of 
Proposition \ref{2-sep} and the structure of the positive cone of Toeplitz matrices.

\begin{corollary}\label{2-sep-cor}
If $a,c\in C(S^1)^{(n)}$ are such that $x=\left[ \begin{array}{cc} a & c^* \\ c & a \end{array}\right]$ is positive, 
then $x$ is separable. Thus,
there exist $k\in\mathbb N$, $\alpha_j\in\mathbb R_+$, and $\lambda_j,\mu_j\in S^1$ such that
\begin{equation}\label{e:dec}
\left[ \begin{array}{cc} a & c^* \\ c & a \end{array}\right] = 
\sum_{j=1}^k \alpha_j \left[ \begin{array}{cc} 1_n & \lambda_j^{-1}T_n(\mu) \\  \lambda_jT_n(\mu) & 1_n \end{array}\right].
\end{equation}
\end{corollary}

\begin{proof} Proposition \ref{2-sep} asserts $x=\displaystyle\sum_{i=1}^q  a_i\otimes b_i$, for some
positive $a_i\in C(S^1)^{(2)}$ and $b_i\in C(S^1)^{(n)}$. By Proposition \ref{cvS}, each 
$a_i$ and $b_i$ is a linear combination of matrices of the
form $T_2(\lambda)$ and $T_n(\mu)$ using nonnegative scalar coefficients. Hence, $x$ can be expressed in the form
(\ref{e:dec}) above.
\end{proof}

Another consequence of $2\times 2$ separability is in reference to 
Toeplitz-matrix moments, which is the problem of
extending positive $2\times 2$ block Toeplitz matrices with Toeplitz blocks
matrices to positive $n\times n$ block Toeplitz matrices with Toeplitz blocks.

\begin{corollary}\label{extension}
If $x_0,x_1$ are $m\times m$ Toeplitz matrices for which
$\left[ \begin{array}{cc} x_0 & x_1^* \\ x_1 & x_0\end{array}\right]$ is positive,
then, for every $n\geq 3$, there exist $m\times m$ Toeplitz matrices $x_2$,\dots, $x_{n-1}$ such that
\[
\left[ \begin{array}{cccc} x_0 & x_{1}^* & \dots & x_{n-1}^* \\
x_1 & x_0 & \ddots & \vdots \\ \vdots & \ddots & \ddots & x_1^* \\
x_{n-1} & \dots & x_1 & x_0 
\end{array}
\right]
\]
is positive.
\end{corollary}

\begin{proof} Corollary \ref{2-sep-cor} shows that
\[
\left[ \begin{array}{cc} x_0 & x_1^* \\ x_1 & x_0\end{array}\right]=\sum_{j=1}^k \alpha_j T_2(\lambda_j)\otimes T_m(\mu_j),
\]
for some $\lambda_j,\mu_j\in S^1$ and $\alpha_j\in\mathbb R_+$. Fix $n\geq 3$ and consider the matrix
\[
X=\sum_{j=1}^k \alpha_j T_n(\lambda_j)\otimes T_m(\mu_j).
\]
Then, $X$ is positive and has the form
\[
X= \left[ \begin{array}{cccc} x_0 & x_{1}^* & \dots & x_{n-1}^* \\
x_1 & x_0 & \ddots & \vdots \\ \vdots & \ddots & \ddots & x_1^* \\
x_{n-1} & \dots & x_1 & x_0 
\end{array}
\right],
\]
for some $x_2$,\dots, $x_{n-1}\in C(S^1)^{(m)}$.
\end{proof} 

As noted by Ozawa in \cite{ozawa2013}, 
a result of Rudin \cite{rudin1963} shows the extension theorem above (Corollary \ref{extension})
fails for $n\times n$ positive block-Toeplitz matrices with Toeplitz blocks, if $n\geq3$. Therefore, we are led to the following
conclusion.

\begin{proposition}\label{non-sep-bttp}
For every $n,m\geq 3$, the cone $(C(S^1)^{(n)} \omin C(S^1)^{(m)})_+$
admits entanglement.
\end{proposition}

\begin{proof}
If not, then, for a fixed $n,m\geq3$ and $x\in (C(S^1)^{(n)} \omin C(S^1)^{(m)})_+$, the matrix $x$ is separable. Thus,
\[
x=\sum_{k=1}^q \alpha_j T_n(\lambda_j)\otimes T_m(\mu_j),
\]
for some $q\in\mathbb N$, $\alpha_j\geq0$, and $\lambda_j,\mu_j\in S^1$. Hence, for any $N>n$, $x$
extends to a positive $N\times N$ block Toeplitz-matrix $\tilde x$ with $m\times m$ 
Toeplitz entries via
\[
\tilde x=\sum_{k=1}^q \alpha_j T_N(\lambda_j)\otimes T_m(\mu_j).
\]
However, this contradicts the results of Ozawa \cite{ozawa2013} and Rudin \cite{rudin1963}.
\end{proof}

%%%%%%%%%%%%%%%%%%%%%%
\subsection{Positive block-Toeplitz matrices with circulant blocks}

\begin{proposition}\label{btcb2} For any $n,m\geq 2$ and $\theta\in\mathbb R$, 
\[
\begin{array}{rcl}
(C(S^1)^{(n)})_+\otimes_{\rm sep} \left(C^{m,\theta}\right)_+
&=& \left(C(S^1)^{(n)}\omin C^{m,\theta} \right)_+ \\
&=& 
(C(S^1)^{(n)})_+\otimes_{{\rm sep}^*} \left(C^{m,\theta}\right)_+
\end{array}.
\]
\end{proposition}

\begin{proof} If $x=[a_{k-j}]_{k,j=1}^{n}\in C(S^1)^{(n)}\omin C^{m,\theta}$ is positive, then
$x\in \M_m(C(S^1)^{(n)})_+$. Therefore,
by Gurvits' Theorem, 
\[
a_\ell=\sum_{j=1}^k \lambda_j^\ell b_j, \mbox{ for every }\ell,
\]
for some $\lambda_j\in S^1$ and $b_j\in\M_m(\mathbb C)_+$. 
Apply the completely positive idempotent map $\mathcal F_{m,\theta}$ of Proposition \ref{idempotent2}  
to each $a_\ell$ to obtain
\begin{equation}\label{e:range of F_m}
a_\ell=\sum_{j=1}^k \lambda_j^\ell \mathcal F_{m,\theta}(b_j), \mbox{ for every }\ell.
\end{equation}
Because the matrices $b_j$ are positive, the generalised
circulant matrices $c_j=\mathcal F_{m,\theta}(b_j)$ are also positive. 
Thus, 
\[
a_\ell=\sum_{j=1}^k \lambda_j^\ell c_j, \mbox{ for every }\ell,
\]
which is equivalent to 
\[
x=\sum_{j=1}^k T_n(\lambda_j)\otimes c_j. 
\]
Hence, $x$ is an element of $(C(S^1)^{(n)})_+\otimes_{\rm sep} \left(C^{m,\theta}\right)_+$.

Now suppose $x$ is an element of $(C(S^1)^{(n)})_+\otimes_{{\rm sep}^*} \left(C^{m,\theta}\right)_+$.
Because of the operator system equivalences $C^{m,\theta}\simeq \cstar(\mathbb Z_m) 
\simeq\mathbb C^m$, the cone $\left(C^{m,\theta}\right)_+$ is affinely isomorphic to $(\mathbb R^m)_+$
and elements of $C(S^1)^{(n)}\otimes C^{m,\theta}$ arise as elements of  $C(S^1)^{(n)}\otimes \mathbb C^m$. However, from
$C(S^1)^{(n)}\omin \mathbb C^m\simeq \displaystyle\bigoplus_1^m C(S^1)^{(n)}$, if $\varphi$ and $\psi$ are
positive linear functionals on $C(S^1)^{(n)}$ and $\mathbb C^m$ respectively, then
$\varphi\otimes\psi$ applied to $x\in C(S^1)^{(n)}\omin \mathbb C^m$ yields
\[
\varphi\otimes\psi(x)=\sum_{j=1}^m \alpha_j\varphi(x_j),
\]
where $\alpha_j=\psi(e_j)$ and $x_j$ is the $j$-th direct summand of $x$. By choosing $j$ and then choosing $\psi$ so that
$\alpha_k=0$ for all $k\not=j$, one gets $\varphi(x_j)\geq 0$ for all positive linear functionals $\varphi$ on $C(S^1)^{(n)}$,
whence $x_j\in \left( C(S^1)^{(n)}\right)_+$. Hence, $\bigoplus_1^m x_j$ is a positive element of 
\[
\displaystyle\bigoplus_1^m C(S^1)^{(n)}\simeq C(S^1)^{(n)}\omin \mathbb C^m,
\]
which proves $(C(S^1)^{(n)})_+\otimes_{{\rm sep}^*} \left(C^{m,\theta}\right)_+\subseteq
\left(C(S^1)^{(n)}\omin C^{m,\theta} \right)_+$.
\end{proof}

%%%%%%%%%%%%%%%%%%%%%%%%%%%%%%%%%%
\subsection{The matrix $R_n$ is separable}

\begin{proposition}\label{R_n is separable} 
$R_n \in C(S^1)^{(n)}_+ \otimes_{\rm sep} C(S^1)^{(n)}_+$.
\end{proposition}

\begin{proof} 
Let $q$ be a prime number such that $q>2n$, and let $\lambda\in S^1$ be a primitive $q$-th root of unity.
If $\ell\in\{-n+1,\dots,n-1\}$ is such that $\ell\not=0$, then $\lambda^\ell$ is also a primitive $q$-th root of unity. Hence
$\displaystyle\sum_{j=1}^q (\lambda^\ell)^j$ is a sum of all the $q$-th roots of unity, and so
$\displaystyle\sum_{j=1}^q (\lambda^\ell)^j=0$. 
Furthermore, assuming $\ell\in\{-n+1,\dots,n-1\}$ and $k\in\{1,\dots,n-1\}$, the
condition $q>2n$ also implies that $\displaystyle\sum_{j=1}^q (\lambda^{k+\ell})^j=0$, if $\ell\not=-k$, and that
$\displaystyle\sum_{j=1}^q (\lambda^{k-\ell})^j=0$, if $\ell\not=k$,

Summing over the primitive roots of unity leads to
\[
\sum_{j=1}^q T_n(\lambda^j)=\sum_{j=1}^q\sum_{\ell=-n+1}^{n-1} (\lambda^j)^\ell r_\ell
=\sum_{\ell=-n+1}^{n-1}\left(\sum_{j=1}^q (\lambda^\ell)^j\right)r_\ell = qr_0.
\]
Next, fix $j\in\{1,\dots,q\}$ and consider $T_n(\lambda^j)\otimes T_n(\lambda^{-j})$, which is the following Toeplitz matrix of Toeplitz matrices:
\[
T_n(\lambda^j)\otimes T_n(\lambda^{-j})=
\left[\begin{array}{cccc}
T_n(\lambda^{-j}) & \lambda^{-j}T_n(\lambda^{-j}) & \dots & \lambda^{-nj+j}T_n(\lambda^{-j}) \\
\lambda^j T_n(\lambda^{-j}) & T_n(\lambda^{-j}) & \ddots & \vdots \\
\vdots & \ddots & \ddots &\lambda^{-j}T_n(\lambda^{-j}) \\
\lambda^{nj-j}T_n(\lambda^{-j}) & \dots & \lambda^j T_n(\lambda^{-j}) & T_n(\lambda^{-j}) 
\end{array}\right].
\]
Note that, for $k=1,\dots,n-1$, 
\[
\lambda^{kj}T_n(\lambda^{-j})=\sum_{\ell=-n+1}^{n-1} \lambda^{kj}(\lambda^{-j})^\ell r_\ell =
\sum_{\ell=-n+1}^{n-1} (\lambda^{k-\ell})^j r_\ell, 
\]
and so
\[
\sum_{j=1}^q\left(\lambda^{kj}T_n(\lambda^{-j})\right)=\sum_{\ell=-n+1}^{n-1} \left(\sum_{j=1}^q(\lambda^{k-\ell})^j\right) r_\ell=
qr_k.
\] 
Similarly, for $k=1,\dots,n-1$,
\[
\sum_{j=1}^q\left(\lambda^{-kj}T_n(\lambda^{-j})\right)=\sum_{\ell=-n+1}^{n-1} \left(\sum_{j=1}^q(\lambda^{k+\ell})^j\right) r_\ell=
qr_{-k}.
\] 
Hence,
\[
\sum_{j=1}^q\frac{1}{q}\left( T_n(\lambda^j)\otimes T_n(\lambda^{-j})\right)
=
\left[\begin{array}{cccc}
r_0& r_{-1} & \dots & r_{-n+1} \\
r_1 & r_0 & \ddots & \vdots \\
\vdots & \ddots & \ddots & r_{-1} \\
r_{n-1} & \dots & r_1 & r_0
\end{array}\right]=R_n,
\]
which proves that $R_n$ is separable.
\end{proof}

The choice of prime $q>2n$ in the proof of Proposition \ref{R_n is separable}, 
while sufficient for the proof of separability, is not claimed to be optimal.
For example, with $n=2$, one can calculate
\[
R_2=\frac{1}{2}\left( T_2(i)\otimes T_2(-i) \,+\, T_2(-i)\otimes T_2(i)\right).
\]

%%%%%%%%%%%%%%%%%%%%%%
\subsection{The tensor cone $\left( C(S^1)^{(n)}\omax C(S^1)^{(m)}\right)_+$}

\begin{theorem}\label{max is sep} For every $n,m\geq 2$,
\[
C(S^1)^{(n)}_+ \otimes_{\rm sep} C(S^1)^{(n)}_+
=
\left( C(S^1)^{(n)}\omax C(S^1)^{(m)}\right)_+.
\]
\end{theorem}

\begin{proof} We first show that $C(S^1)^{(n)}_+ \otimes_{\rm sep} C(S^1)^{(n)}_+$ is topologically closed.
Consider the compact set $\mathcal E= \left\{  T_n(\lambda)\otimes T_m(\mu))\,|\,\lambda,\mu\in S^1 \right\}$.
Because the extremal rays of $C(S^1)^{(n)}_+ \otimes_{\rm sep} C(S^1)^{(n)}_+$ have the form
\[
\left\{ \alpha( T_n(\lambda)\otimes T_m(\mu))\,|\,\alpha\geq0,\,\lambda,\mu\in S^1 \right\},
\]
the set $\mathcal E$ constitutes the set of extreme points of the convex hull $\mathcal C$ of $\mathcal E$.
By Carath\'eodory's Theorem, the convex hull of a compact set in a finite-dimensional vector
space is compact; thus,  
$\mathcal C$ is compact. But $\mathcal C$ is also the base for the cone
$C(S^1)^{(n)}_+ \otimes_{\rm sep} C(S^1)^{(n)}_+$, and so $C(S^1)^{(n)}_+ \otimes_{\rm sep} C(S^1)^{(n)}_+$
is closed.

Select $x\in \left( C(S^1)^{(n)}\omax C(S^1)^{(m)}\right)_+$. By definition, for each
$\varepsilon>0$, the matrix $x+\varepsilon(1_n\otimes 1_m)$ has the form 
$\gamma^*(p\otimes q)\gamma$, for some positive matrices $p$ and $q$ with entries from
$C(S^1)^{(n)}$ and $C(S^1)^{(m)}$, respectively, and some rectangular complex matrix $\gamma$.
Because  $x+\varepsilon(1_n\otimes 1_m)$ is strictly positive in $C(S^1)^{(n)}\omax C(S^1)^{(m)}$,
there are,
by \cite[Proposition 3.4]{farenick--mcburney2023},
$k_\varepsilon\in\mathbb N$, $a_j^{(\varepsilon)}\in C(S^1)^{(n)}_+$, and $b_j^{(\varepsilon)}\in C(S^1)^{(m)}_+$
such that
and so, by \cite[Proposition 3.4]{farenick--mcburney2023} it may be expressed as
\[
x+\varepsilon(1_n\otimes 1_m)= \sum_{j=1}^{k_\varepsilon} a_j^{(\varepsilon)} \otimes b_j^{(\varepsilon)}.
\]
Thus, $x+\varepsilon(1_n\otimes 1_m)$ is separable, for every $\varepsilon>0$. Because the separability cone is
closed, we deduce that $x$ itself is separable.
\end{proof}

%%%%%%%%%%%%%%%%%%%%%%%%%%%%%%
\section{Entanglement}

If $\osr$ and $\ost$ are finite-dimensional operator systems, then
\[
\osr_+\otimes_{{\rm sep}^*} \ost_+=\left\{x\in(\osr\otimes\ost)_{\rm sa}\,|\,(\varphi\otimes\psi)[x]\geq 0, \mbox{ for all }\varphi\in\osr^d_+,\,\psi\in\ost^d_+\right\}.
\]
By the convexity of $\ost^d_+$ and $\osr^d_+$, one may restrict themselves to those $\varphi$ and $\psi$ that generate the extremal rays of these cones.
Fortunately, these extremal rays have been determined by Connes and van Suijlekom for both Toeplitz and Fej\'er-Riesz operator systems
\cite[Propositions 4.5 and 4.8]{connes-vansuijlekom2021}.

 \begin{proposition}[Extremal Rays of the Dual]\label{ext rays dual}
 For every $n\geq 2$: 
 \begin{enumerate}
 \item the extremal positive linear functionals on $C(S^1)^{(n)}$ are those of the form $x\mapsto \langle x\xi,\xi\rangle$, where 
 $\xi=\left[ \begin{array}{c} \xi_0 \\ \xi_1 \\ \vdots \\ \xi_{n-1}\end{array}\right]\in\mathbb C^n$ is a vector of coefficients in which the polynomial
 \[
 f_\xi(z)=\sum_{\ell=0}^{n-1} \xi_\ell z^{n-1-\ell}
 \]
 has all of its roots on the unit circle $S^1$;
 \item the extremal positive linear functionals on $C(S^1)_{(n)}$ are positive scalar multiples of the point-evaluations
 $f\mapsto f(\lambda)$, for $\lambda\in S^1$.
 \end{enumerate}
 \end{proposition}

 %%%%%%%%%%%%%%%%%%%%%%%%%%
 \subsection{Entangled block Toeplitz matrices}
 
If $\xi\in\mathbb C^2$ is such that $f_\xi(\lambda)=0$ only for $\lambda\in S^1$, then scaling $\xi$ by an appropriate positive scalar multiple allows us to 
assume $\xi$ is an element of $S^1\times S^1\subset\mathbb C^2$. Thus, these are the vector states that generate the extremal
 rays of the cone of positive linear functionals on $C(S^1)^{(2)}$.

 \begin{lemma}\label{snow} If $a=\left[\begin{array}{cc} \alpha & \overline\beta \\ \beta & \alpha\end{array}\right]\in C(S^1)^{(2)}$ satisfies
 $\langle a\xi,\xi\rangle\geq 0$ for every $\xi\in S^1\times S^1$,
 then $\langle a\gamma,\gamma\rangle\geq 0$ for every $\gamma\in \mathbb C^2$.
 \end{lemma}
 
 \begin{proof} In writing $\xi=\lambda e_1+\mu e_2$, for $\lambda,\mu\in S^1$, we obtain
 \[
 0\leq \langle a\xi,\xi\rangle = 2\alpha + 2\Re( \lambda\overline\mu\beta).
 \]
 Thus, $\langle a\xi,\xi\rangle\geq 0$ for every $\xi\in S^1\times S^1$ if and only if $\alpha\geq-\Re(e^{i\theta}\beta)$, for every $\theta\in\mathbb R$.
 In selecting $\theta$ so that $-\Re(e^{i\theta}\beta)=|\beta|$, one obtains $\alpha\geq|\beta|$, which implies $a$ is a positive operator
 on $\mathbb C^2$.
 \end{proof}
 
 \begin{proposition}\label{d-sep2x2} The following statements are equivalent for $x=\left[\begin{array}{cc} a& b^* \\ b & a\end{array}\right]$
 in $C(S^1)^{(2)}\otimes C(S^1)^{(2)}$:
 \begin{enumerate}
 \item $x\in  (C(S^1)^{(2)})_+\otimes_{{\rm sep}^*}  C(S^1)^{(2)})_+$;
 \item $\frac{1}{2}(e^{i\theta} b + e^{-i\theta}b^*)\leq a$, for all $\theta\in \mathbb R$.
 \end{enumerate}
 \end{proposition}
 
 \begin{proof} 
 In fixing $\eta\in S^1\times S^1$ and taking 
 $\xi=\lambda e_1 + \mu e_2$, for $\lambda,\mu\in S^1$,
 we obtain
 \[
 \langle x(\xi\otimes\eta), (\xi\otimes\eta)\rangle = 2\left(\langle a\eta,\eta\rangle + \Re(\lambda\overline\mu\langle b\eta,\eta\rangle\right).
 \]
 Hence, $x\in  (C(S^1)^{(2)})_+\otimes_{{\rm sep}^*}  C(S^1)^{(2)})_+$ if and only if, for each $\theta\in\mathbb R$,
 \[
 \left\langle \left(a + \Re(e^{i\theta}b)\right)\eta,\eta\right\rangle \geq 0,  \,\mbox{ for all }\,\eta\in S^1\times S^1.
 \]
 However, by Lemma \ref{snow}, the condition above holds if and only if $ a + \Re(e^{i\theta}b) $ is a positive operator on $\mathbb C^2$, for every
 $\theta\in\mathbb R$.
 \end{proof}
 
 Note $(C(S^1)^{(2)})_+\otimes_{{\rm sep}^*}  C(S^1)^{(2)})_+$ properly contains
 $\left( C(S^1)^{(2)}\omin C(S^1)^{(2)}\right)_+$  because, if they were equal, then 
 Proposition \ref{2-sep} would imply every element of the cone
 $(C(S^1)^{(2)})_+\otimes_{{\rm sep}^*}  C(S^1)^{(2)})_+$
 is separable, contrary to the conclusion of Corollary \ref{nope}

 %%%%%%%%%%%%%%%%%%%%%%
 \subsection{Entanglement in tensor products of Fej\'er-Riesz cones}
 
 \begin{proposition}\label{very max} For all $n,m\geq 2$, 
 \[
 \left( C(S^1)_{(n)}\omin C(S^1)_{(m)}\right)_+ = \left( C(S^1)_{(n)}\right)_+ \otimes_{{\rm sep}^*} \left( C(S^1)_{(m)}\right)_+
 \]
 and 
 \[
 \left( C(S^1)_{(2)}\omax C(S^1)_{(m)}\right)_+ = \left( C(S^1)_{(2)}\right)_+ \otimes_{{\rm sep}^*} \left( C(S^1)_{(m)}\right)_+.
 \]
 \end{proposition}

\begin{proof} If $\lambda,\mu\in S^1$ and if $\varrho_\lambda$ and $\varrho_\mu$ are the corresponding point evaluations, then, for every
$f\in \left( C(S^1)_{(n)}\right)_+ \otimes_{{\rm sep}^*} \left( C(S^1)_{(m)}\right)_+$, 
\[
0 \leq (\varrho_\lambda\otimes\varrho_\mu)[f] = f(\lambda,\mu);
\]
hence, $f\in \left( C(S^1)_{(n)}\omin C(S^1)_{(m)}\right)_+$.

Proposition \ref{2-sep} show that the identity map
\[
\iota: C(S^1)^{(2)}\omin C(S^1)^{(m)} \rightarrow C(S^1)^{(2)}\omax C(S^1)^{(m)}
\]
is a positive linear map of operator systems. Hence, the dual map $\iota^d$ is also positive, 
where
\[
\iota^d: \left(C(S^1)^{(2)}\omax C(S^1)^{(m)}\right)^d \rightarrow \left(C(S^1)^{(2)}\omin C(S^1)^{(m)}\right)^d.
\]
However, in identifying $(\osr\omin\ost)^d$ with $\osr^d\omax\ost^d$, by Theorem \ref{tensor duality}, then Toeplitz duality 
yields the positivity of the map
\[
\iota^d: C(S^1)_{(2)}\omin C(S^1)_{(m)} \rightarrow C(S^1)_{(2)}\omax C(S^1)_{(m)},
\]
which proves $\left( C(S^1)_{(2)}\omax C(S^1)_{(m)}\right)_+ = \left( C(S^1)_{(2)}\right)_+ \otimes_{{\rm sep}^*} \left( C(S^1)_{(m)}\right)_+$.
\end{proof}

The separability cones for Fej\'er-Riesz operator systems are really quite small, since, by definition, 
$f\in\left( C(S^1)_{(n)}\right)_+ \otimes_{{\rm sep}} \left( C(S^1)_{(m)}\right)_+$
if and only if there are $g_j\in \left( C(S^1)_{(n)}\right)_+$ and $h_j\in\left( C(S^1)_{(m)}\right)_+$ such that
\[
f(z_1,z_2)=\sum_{j=1}^k g_j(z_1)h_j(z_2).
\]
More natural, and in the spirit of the single-variable Fej\'er-Riesz factorisation of positive trigonometric polynomials $f(z)$
as $f(z)=|h(z)|^2$ for a suitable polynomial $f\in\mathbb C[z]$, are representations 
of $f\in \left( C(S^1)_{(n)}\omin C(S^1)_{(m)}\right)_+$ as sums of squares, by which is meant
\[
f(z_1,z_2)=\sum_{j=1}^k \overline{h_j(z_1,z_2)}h_j(z_1,z_2) = \sum_{s=1}^k|h_j(z_1,z_2)|^2,
\]
for some polynomials $h_j\in\mathbb C[z_1,z_2]$.

 %%%%%%%%%%%%%%%%%%%%%%
 \subsection{The matrix $T_n$ is entangled}
 
 \begin{proposition}\label{T_n is entangled}
 $T_n$ is entangled in $\left(C(S^1)^{(n)}\omin C(S^1)_{(n)} \right)_+$.
 \end{proposition}
 
 \begin{proof}  By \cite[Theorem 4.4]{farenick--mcburney2023}, the matrix $T_n$ is entangled in 
 $\left(\M_n(\mathbb C)\omin C(S^1)_{(n)} \right)_+$; thus, $T_n$ is entangled in the smaller cone
 $\left(C(S^1)^{(n)}\omin C(S^1)_{(n)} \right)_+$.
 \end{proof}

%%%%%%%%%%%%%%%%%%%%%%%%%%%%%%%%%%%%%%%%%%%%%%%%%%%%%%%%%%%%%%%%%%%%
\section{Categorical Relations}

%%%%%%%%%%%%%%%%%%%%%%%%%%
\subsection{C$^*$- and injective envelopes}

The operator system category contains 
all unital C$^*$-algebras as objects. For each operator system $\osr$, there is a unital C$^*$-algebra $\cstare(\osr)$,
called the \emph{C$^*$-envelope of $\osr$}, and unital linear complete order embedding 
$\iota_{\rm e}:\osr\rightarrow\cstare(\osr)$ in which the operator subsystem $\iota_{\rm e}(\osr)$ 
of $\cstare(\osr)$ generates $\cstare(\osr)$, such that the pair $\left(\cstare(\osr),\iota_{\rm e}\right)$ has the following 
universal property:
for every unital C$^*$-algebra $\A$ and every unital linear complete order
embedding $\phi:\osr\rightarrow\A$ 
for which $\phi(\osr)$ generates $\A$, there exists a unital surjective $*$-homomorphism $\pi:\A\rightarrow\cstare(\osr)$
such that $\iota_e=\pi\circ\phi$. The existence of the injective envelope was established by Hamana \cite{hamana1979b},
and it serves as a type of noncommutative \v Silov boundary for operator systems. 
One of the main tasks in the theory of operator systems is to identify $\cstare(\osr)$ for various operator systems $\osr$.

\begin{definition} An operator system $\osi$ is \emph{injective} if, for each operator system $\osr$,
ucp map $\alpha:\osr\rightarrow\osi$, and unital linear complete order embedding $\beta:\osr\rightarrow\ost$
of $\osr$ into an operator system $\ost$, there exists a ucp map $\phi:\ost\rightarrow\osi$ such that
$\alpha=\phi\circ\beta$.
\end{definition}

The Arveson Extension Theorem \cite{arveson1969,Paulsen-book} establishes the injectivity of the
type I factor $\B(\H)$, for any Hilbert space $\H$, which is the most most basic of injective operator systems.
Choi and Effros \cite{choi--effros1977} developed an abstract approach to injectivity, which Hamana used to
establish the existence of the injective envelope of operator systems \cite{hamana1979b}. 
For each operator system $\osr$, there is an injective operator system $\osi(\osr)$,
called the \emph{injective envelope of $\osr$}, and unital linear complete order embedding 
$\iota_{\rm ie}:\osr\rightarrow\osi(\osr)$ with the property that a sequence of operator subsystems of the form
\[
\iota_{\rm ie}(\osr)\subseteq \osi_0 \subseteq \osi,
\]
with $\osi_0$ injective, can occur only if $\iota_0=\osi$.

The relationship between C$^*$-envelopes and injective envelopes is captured by the following theorem
\cite[Chapter 15]{Paulsen-book}.

\begin{theorem} If $\osi$ is an injective operator system, then there is a unital C$^*$-algebra $\B$
such that $\osi\simeq \B$. Furthermore, for every operator system $\osr$, there are 
an injective unital C$^*$-alegbra $\B$, a unital C$^*$-subalgebra $\A\subseteq\B$, and a unital 
linear complete order embedding $\kappa:\osr\rightarrow\A$
such that
\begin{equation}\label{e:hamana1}
\cstar\left(\kappa(\osr)\right)=\A\simeq\cstare(\osr)
\end{equation}
and
\begin{equation}\label{e:hamana2}
\osr\simeq\kappa(\osr) \subseteq \A \subseteq \B \simeq \osi(\osr),
\end{equation}
where, if $\phi:\B\rightarrow\osi(\osr)$ is a unital linear complete order isomorphism implementing 
the categorical isomorphism $\B \simeq \osi(\osr)$, then $\iota_{\rm ie}=\phi\circ\kappa$.
\end{theorem} 

Hamana's work on injective envelopes \cite{hamana1979b} 
shows that there is an operator subsystem $\mathcal Q$ of $\osi(\osr)$ such that $\cstare(\osr)\simeq\mathcal Q$ and 
$\osr\subseteq\mathcal Q\subseteq\osi(\osr)$. If $\osr$ and $\cstare(\osr)$ have finite dimension, then,
because finite-dimensional C$^*$-algebras are injective operator systems, 
it must be that $\mathcal Q=\osi(\osr)$. Hence,
$\cstare(\osr)\simeq\osi(\osr)$, for every operator system $\osr$ in which both
$\osr$ and $\cstare(\osr)$ have finite dimension.
This fact makes it relatively easy to compute the injective envelopes of irreducible operator systems of matrices.

\begin{proposition}\label{not-cstar} If $\osr$ is an operator subsystem of $\M_n(\mathbb C)$ such that $\osr\not=\M_n(\mathbb C)$
and $\osr'\simeq\mathbb C$, then 
\begin{enumerate}
\item there is no unital C$^*$-algebra $\A$ for which $\osr\simeq\A$,
\item $\osr$ is not an injective operator system, and 
\item the C$^*$- and injective envelopes of $\osr$ are $\M_n(\mathbb C)$.
\end{enumerate}
\end{proposition} 

\begin{proof}
If, on the contrary, $\osr\simeq\A$ for unital C$^*$-algebra $\A$, then $\A$ has finite dimension and is, therefore,
isomorphic to a finite direct sum of type I factors. By Arveson's Extension Theorem \cite{arveson1969}, type I factors are injective; hence,
$\A$ and, thus, $\osr$ are injective. With $\osr$
being injective, there must exist an idempotent ucp map $\Phi:\M_n(\mathbb C)\rightarrow\M_n(\mathbb C)$ for which
the range of $\Phi$ is $\osr$. Hence, the fixed point set of $\Phi$ contains $\osr$. Because $\osr'\simeq\mathbb C$, the fixed point
set of $\Phi$ is an irreducible operator subsystem of $\M_n(\mathbb C)$. By Arveson's Boundary Theorem
\cite{arveson1972}, $\Phi$ is necessarily the identity map, and so the range of $\Phi$ cannot be a proper subset
of $\M_n(\mathbb C)$. Hence, it cannot be that $\osr$ is injective or that $\osr\simeq\A$, proving (1) and (2).

To prove (3), 
recall that the C$^*$-envelope $\cstare(\osr)$ of $\osr$ is a quotient of the 
$C^*$-algebra generated by $\osr$. Because $\osr$ is a matrix operator system, $\cstar(\osr)=\osr''=\M_n(\mathbb C)$, which is a
simple C$^*$-algebra. Hence, $\cstare(\osr)=\M_n(\mathbb C)$ and, by the remarks preceding the statement of
the Proposition, $\cstare(\osr)=\osi(\osr)$.
\end{proof}

\begin{corollary} The conclusions of Proposition \ref{not-cstar} apply to Toeplitz operator systems.
\end{corollary}

\begin{proof} For each $n\geq 2$, the commutant of $C(S^1)^{(n)}$ in $\M_n(\mathbb C)$ is contained in the commutant of 
$\{r_1,r_{-1}\}$, which is simply $\{\alpha 1_n\,|\,\alpha\in\mathbb C\}$.
\end{proof} 

For Fej\'er-Riesz operator systems, a result similar to Proposition \ref{not-cstar} holds.

\begin{proposition}\label{poi} There is no unital C$^*$-algebra $\A$ for which $C(S^1)_{(n)}\simeq\A$, if $n\geq 2$.
\end{proposition}

\begin{proof} If there were a unital C$^*$-algebra $\A$ for which $C(S^1)_{(n)}\simeq\A$, then the finite-dimensionality
of $\A$ would yield the injectivity of $C(S^1)_{(n)}$, implying that the C$^*$-envelope and injective envelope of $C(S^1)_{(n)}$
were also $\A$. However, the C$^*$-envelope of $C(S^1)_{(n)}$ is $C(S^1)$ \cite[Proposition 4.3]{connes-vansuijlekom2021},
which has infinite dimension.
\end{proof}

As noted in the proof of Proposition \ref{poi}, the C$^*$-envelopes of Fej\'er-Riesz operator systems
have been determined by Connes and van Suijlekom \cite[Proposition 4.3]{connes-vansuijlekom2021}, where they proved 
$\cstare\left( C(S^1)_{(n)}\right)= C(S^1)$, for every $n\geq 2$. The injective envelopes of
Fej\'er-Riesz operator systems are also abelian C$^*$-algebras, but with extremely disconnected spectra. 

\begin{lemma}\label{ie} If $\Delta_{S^1}=\displaystyle\lim_\leftarrow \beta X$, the topological inverse limit of the 
Stone-\v Cech compactifications $\beta X$, partially ordered by reverse inclusion,
of dense open subsets $X \subseteq S^1$, 
then $\Delta_{S^1}$ is a compact, Hausdorff, and
extremely disconnected topological space, and
\[
\osi\left(C(S^1)_{(n)}\right) \simeq C(\Delta_{S^1}),
\]
for every $n\geq 2$.
\end{lemma}

\begin{proof} Fix $n\geq 2$, and consider the canonical embedding of the operator system
$C(S^1)_{(n)}$ into its C$^*$-envelope $C(S^1)$. As an operator system and its C$^*$-envelope 
have the same injective envelope, it is enough to consider the injective envelope of $C(S^1)$.

Let
$O_d(S^1)=\{X\subseteq S^1\,|\,X\,\mbox{ is open and dense in }\,S^1\}$.
For each $X\in O_d(S^1)$, the set 
$C_0(X)$ of continuous complex-valued functions on $X$
that vanish at infinity is an essential ideal of the algebra $C(S^1)$,
and $C(\beta\Omega)$ is the multiplier algebra of this essential ideal,
where $\beta X$ denotes the Stone-\v Cech compactification of $X$.
Further, 
if $\iota_X:X\rightarrow \beta X$ denotes the
continuous embedding of $X$ as a dense subset of $\beta X$, then, because each
$X\in O_d(S^1)$ is open and, hence, locally compact \cite[Theorem XI.6.5]{Dugundji-book},
the embedding $\iota_X:X\rightarrow\beta X$ is an open map \cite[Theorem VII.7.3]{Dugundji-book}.
Therefore, $\iota_X(X)$ is a dense open subset of $\beta X$.
If
$X,Z\in O_d(S^1)$ satisfy $X\subset Z$, then $\iota_Z$ embeds $X$
into $\beta Z$  as a dense subset. Thus, $\beta Z$ is a
compactification of $X$ and so, by the Stone--\v Cech Theorem
\cite[Theorem 8.2]{Dugundji-book}, there is a unique continuous
function $\Phi_{Z,X}:\beta X\rightarrow \beta Z$ for which
$\Phi_{Z,X}\circ\iota_X=\iota_Z|_X$. Because $\iota_Z(X)$ is dense
in $\beta Z$, $\Phi_{Z,X}$ is a surjection. Note that if $X\subset
W\subset Z$, for $X,W,Z\in O_d(S^1)$, then
$\Phi_{Z,X}=\Phi_{Z,W}\circ\Phi_{W,X}$. Hence, $(\{\beta X\,:\,X\in
O_d(S^1)\},\, \Phi_{Z,X})$ is an inverse spectrum over $O_d(S^1)$
endowed with the order of reversed inclusion. Thus, if $\Delta_{S^1}=\displaystyle\lim_{\leftarrow}\,\beta X$,
then, by \cite{semadeni1968},
\[
C(\Delta_{S^1})=C\left( \lim_{\leftarrow}\,\beta X\right)=\lim_{\rightarrow} C(\beta X).
\]

Because the direct limit C$^*$-algebra 
$\displaystyle\lim_{\rightarrow} C(\beta X)$ is the local multiplier algebra \cite{Ara--Mathieu-book} of $C(S^1)$, 
and because the local multiplier algebra of any unital C$^*$-algebra $\A$ is an operator subsystem of 
the injective envelope of $\A$ \cite[Theorem 4.5]{frank--paulsen2003}, we deduce that $C(\Delta_{S^1})$ is an operator subsystem of the
injective envelope of $C(S^1)_{(n)}$ and contains $C(S^1)_{(n)}$ as an operator subsystem. 
However, as $C(\Delta_{S^1})$ is an abelian AW$^*$-algebra \cite[Proposition 3.1.5]{Ara--Mathieu-book},
$C(\Delta_{S^1})$ is injective, and so it must coincide with the injective envelope of $C(S^1)_{(n)}$.
Because the maximal ideal space of an abelian AW$^*$-algebra is compact, Hausdorff, and
extremely disconnected, the proof is complete.
\end{proof}

%%%%%%%%%%%%%%%
\subsection{Tensor relations for Toeplitz and Fej\'er-Riesz operator systems}

Some tensorial properties of $C(S^1)^{(n)}$ and $C(S^1)_{(n)}$ can be deduced directly from previous work, 
as indicated in the proposition below.

\begin{proposition}\label{non-equal and injective} 
Assuming $n,m\in\mathbb N$ are at least $2$, we have
\[
\begin{array}{rcl}
C(S^1)^{(n)}\omin C(S^1)^{(m)}&\not =& C(S^1)^{(n)}\omax C(S^1)^{(m)}, \mbox{ and } \\
C(S^1)_{(n)}\omin C(S^1)_{(m)}&\not =& C(S^1)_{(n)}\omax C(S^1)_{(m)}. 
\end{array}
\]
However, if $\ost$ is an injective operator system, then
\[
\begin{array}{rcl}
C(S^1)^{(n)}\omin\ost &=& C(S^1)^{(n)}\omax\ost 
\mbox{ and } \\
C(S^1)_{(n)}\omin\ost &=& C(S^1)_{(n)}\omax\ost.
\end{array}
\]
In particular,
\[
\begin{array}{rcl}
C(S^1)^{(n)}\omin\B(\H) &=& C(S^1)^{(n)}\omax\B(\H) 
\mbox{ and } \\
C(S^1)_{(n)}\omin\B(\H) &=& C(S^1)_{(n)}\omax\B(\H),
\end{array}
\]
for every Hilbert space $\H$.
\end{proposition}

\begin{proof} The first two assertions are established in \cite{farenick2021}.

Because $C(S^1)^{(n)}$ and $C(S^1)_{(n)}$ are operator subsystems of nuclear C$^*$-algebras, they 
are also exact operator systems and, therefore, have the local lifting property \cite[Theorem 6.6]{kavruk2014}.
Hence, the asserted tensorial properties hold by
\cite[Theorem 8.6]{kavruk--paulsen--todorov--tomforde2013}.
Alternatively, the exactness of
$C(S^1)^{(n)}$ and $C(S^1)_{(n)}$ imply they have the local lifting property \cite[Theorem 6.6]{kavruk2014}, which is equivalent
to being CP-stable \cite[Theorem 7.7]{goldbring--sinclair2017}, 
\cite[Proposition 2.5]{sinclair2017}. Thus, 
applying Theorem B of \cite{sinclair2017} to the CP-stable operator systems $C(S^1)^{(n)}$ and $C(S^1)_{(n)}$,
the asserted tensorial properties $C(S^1)^{(n)}\omin\B(\H) = C(S^1)^{(n)}\omax\B(\H)$
and $C(S^1)_{(n)}\omin\B(\H) = C(S^1)_{(n)}\omax\B(\H)$ hold.

A general property of an operator system $\osr$ 
for which $\osr\omin\B(\H) = \osr\omax\B(\H)$, for every Hilbert space $\H$, is that, for every unital C$^*$-algebra $\B$, 
a matrix $X\in \M_p(\osr\omin\B)$ is positive if and only if $X$ is a positive matrix in $\M_p(\osr\omax\osi(\B))$
\cite[Theorem 8.1]{kavruk--paulsen--todorov--tomforde2013}, where $\osi(\B)$ is the injective envelope of $\B$.
Thus, if $\ost$ is an injective operator system, then 
it coincides with its injective envelope and a 
matrix $X\in \M_p(\osr\omin\ost)$ is positive if and only if $X$ is a positive matrix in $\M_p(\osr\omax\ost)$, which is to say that
$\osr\omin\ost = \osr\omax\ost$. Because $C(S^1)^{(n)}$ and $C(S^1)_{(n)}$ satisfy the tensorial property with $\B(\H)$, we 
conclude that $C(S^1)^{(n)}$ and $C(S^1)_{(n)}$ satisfy the same property when tensored with any injective operator system.
\end{proof}

Recall from basic linear algebra
that if $\B$ is an algebra
and if $\phi:\osr\rightarrow\B$ and $\psi:\ost\rightarrow\B$ are linear maps of vector spaces $\osr$ and $\ost$ with commuting ranges
(i.e., $\phi(x)\psi(y)=\psi(y)\phi(x)$ for all $x\in\osr$, $y\in\ost$), then there exists a unique linear map
$\osr\otimes\ost\rightarrow\B$ in which $x\otimes y\mapsto \phi(x)\psi(y)$. This linear
map is denoted here by $\phi\cdot\psi$. 

\begin{definition}\label{oc} The \emph{commuting operator system tensor product}, $\oc$, 
of operator systems $\osr$ and $\ost$ is the operator system structure on the algebraic tensor product
$\osr\otimes\ost$ obtained by
declaring a matrix $x\in\M_p(\osr\otimes\ost)$ to be positive 
if $(\phi\cdot\psi)^{(p)}[x]$ is a positive element of $\M_p(\B(\H))$, for every Hilbert space $\H$ and every
pair of ucp maps
$\phi:\osr\rightarrow\B(\H)$ and $\psi:\ost\rightarrow\B(\H)$ with commuting ranges.
\end{definition}

The following definition is motivated by the defining conditions for the maximally entangled state in quantum theory.

\begin{definition} Suppose that $\{x_1,\dots,x_n\}$ is a linear basis for an operator system $\osr$,
and that $\{\varphi_1,\dots,\varphi_n\}$ is the associated dual basis for the operator system dual $\osr^d$. The element
\[
\xi=\sum_{j=1}^n x_j\otimes\varphi_j
\]
is called the \emph{maximally entangled positive element} of $\osr\omin\osr^d$.
\end{definition}

The terminology above is due to Kavruk \cite{kavruk2015}, who also proves the maximally entangled element $\xi$  is positive and that its definition
is independent of the choice of linear basis of $\osr$. However, Kavruk does not actually prove $\xi$ is entangled, and so this is explained in Proposition \ref{maxent}
below for Toeplitz and Fej\'er-Riesz operator systems.

In the case of the operator system $C(S^1)^{(n)}\omin C(S^1)_{(n)}$, the maximally
entangled positive element is represented with respect to the canonical linear bases by
\[
\xi_n=\sum_{\ell=-n+1}^{n-1}r_\ell\otimes \chi_{-\ell}.
\]

\begin{proposition}\label{maxent} $\xi_n$ is entangled in $\left(C(S^1)^{(n)}\omin C(S^1)_{(n)}\right)_+$.
\end{proposition}

\begin{proof} The linear map $\phi_n:C(S^1)_{(n)}\rightarrow C(S^1)_{(n)}$ defined on the basis elements of $C(S^1)_{(n)}$
by $\phi_n(\chi_\ell)=\chi_{-\ell}$ has the property that
\[
\left(\mbox{\rm Id}_{C(S^1)^{(n)}}\otimes \phi_n\right)[\xi_n] = T_n,
\]
the universal positive Toeplitz matrix positive. The positivity of $T_n$ implies
the complete positivity of $\phi_n$, by Proposition \ref{pos-cp}; therefore, if $\xi_n$ were
separable, then $T_n$ would be too, contrary to the conclusions of Proposition \ref{T_n is entangled}.
Hence, $\xi_n$ must be entangled.
\end{proof}

Because the C$^*$-envelopes of Toeplitz and Fej\'er-Riesz operator systems are nuclear C$^*$-algebras,  
the following result is a useful tool for determining their categorical properties.

\begin{theorem}\label{nuc prop} The following statements are equivalent for a finite-dimensional operator system
$\osr$ for which $\cstare(\osr)$ is a nuclear C$^*$-algebra:
\begin{enumerate}
\item $\osr\omin\A=\osr\omax\A$, for every unital C$^*$-algebra $\A$;
\item $\osr\omin\cstar(\mathbb F_\infty)=\osr\omax\cstar(\mathbb F_\infty)$;
\item $\osr\omin\ost=\osr\oc\ost$, for every operator system $\ost$;
\item for every unital complete order embedding $\kappa:\osr\rightarrow\B(\H)$, there exists a ucp map
$\phi:\B(\H)\rightarrow\kappa(\osr)''$ such that $\phi\circ\kappa=\kappa$;
\item for every unital complete order embedding $\kappa:\osr\rightarrow\B(\H)$, there exists a ucp map
$\phi:\osi(\osr)\rightarrow\kappa(\osr)''$ such that $\phi(x)=\kappa(x)$, for every $x\in \osr$;
\item the maximally entangled element $\xi\in \osr\otimes\osr^d$ is positive in $\osr\oc\osr^d$.
\end{enumerate}
\end{theorem}

\begin{proof} Because $\cstare(\osr)$ is a nuclear C$^*$-algebra containing $\osr$ as an operator subsystem, the
operator system $\osr$ is exact in the category $\mathfrak S_1$ \cite[Proposition 4.10]{kavruk2014}. Hence,
$\osr\omin\ost=\osr\otimes_{\rm el}\ost$, for every operator system $\ost$ \cite[Theorem 5.7]{kavruk--paulsen--todorov--tomforde2013},
where $\otimes_{\rm el}$ is the operator system structure on $\osr\otimes\ost$ in which a matrix $X\in\M_p(\osr\otimes\ost)$
is positive if $X$ is positive in $\M_p(\osi(\osr)\omax\ost)$. By \cite[Theorem 7.3]{kavruk--paulsen--todorov--tomforde2013}, statements
(4) and (5) are equivalent, and each of these is
equivalent to the statement that $\osr\otimes_{\rm el}\ost=\osr\oc\ost$, for all operator systems $\ost$.
Hence, for operator systems in which $\cstare(\osr)$ is nuclear, we have (4) or (5) if and only if 
$\osr\omin\ost=\osr\otimes_{\rm el}\ost=\osr\oc\ost$ for all operator systems $\ost$, thereby showing the equivalence of (4), (5), and (3)
for operator systems $\osr$ in which $\cstare(\osr)$ is nuclear. The equivalence of statements (3) and (1) is given by \cite[Proposition 4.11]{kavruk2014},
and the equivalence of (4) and (2) is given by \cite[Theorem 7.6]{kavruk--paulsen--todorov--tomforde2013}. Finally, using for the first time the finite-dimensionality
of $\osr$, the equivalence of (1) and (6) is given by \cite[Theorem A.1]{kavruk2015}.
\end{proof}

It is perhaps worth noting explicitly that, as shown by the proof of Theorem \ref{nuc prop}, statements (1) through (5) are equivalent whenever $\cstare(\osr)$
is nuclear, regardless of whether $\osr$ has finite dimension or not.

\begin{proposition}\label{qaz} If $\osr$ is a $3$-dimensional Toeplitz or Fej\'er-Riesz operator system, then $\osr$ satisfies equivalent
condition (5) of Theorem \ref{nuc prop}.
\end{proposition} 

\begin{proof} Assume $\osr=C(S^1)^{(2)}$. We are to prove:
for every unital complete order embedding $\kappa:\osr\rightarrow\B(\H)$, there exists a ucp map
$\phi:\osi(\osr)\rightarrow\kappa(\osr)''$ such that $\phi(x)=\kappa(x)$, for every $x\in \osr$.
Because the injective envelope of $\osr$ is $\M_2(\mathbb C)$ and $\kappa(\osr)''$ is a von Neumann algebra,
a theorem of Haagerup \cite[Remark 3.7]{haagerup1983} establishes the extension $\phi$ of $\kappa$ with the
required co-domain $\kappa(\osr)''$. 

Since $\osr=C(S^1)^{(2)}$ satisfies equivalent
condition (5) of Theorem \ref{nuc prop}, so does $\osr^d=C(S^1)_{(2)}$, by equivalent condition (6).
\end{proof} 

Equivalent condition (1) in Theorem \ref{nuc prop}
has been noted by Kavruk in \cite{kavruk2014,kavruk2015} in the case of the $3$-dimensional Toeplitz or
Fej\'er-Riesz operator system, and so Proposition \ref{qaz} offers an alternative proof of Kavruk's results.

%%%%%%%%%%%%%%%%%%%%
\subsection{The weak expectation property}

An operator system $\osr$ has the weak expectation property if the canonical embedding of $\osr$ into its bidual $\osr^{dd}$
has a ucp extension $\osi(\osr)\rightarrow \osr^{dd}$. The weak expectation property is motivated by operator algebra theory, where the property has an
important role in a number of C$^*$-algebraic problems. However, in the operator system category $\mathfrak S_1$, it seems rather difficult for an operator system
that is not already a C$^*$-algebra to possess this property. The results in this subsection confirm neither Toeplitz or Fej\'er-Riesz operator systems
possess this property. 

\begin{proposition}\label{no-wep1} If $\osr$ is an operator subsystem of $\M_n(\mathbb C)$ such that $\osr\not=\M_n(\mathbb C)$
and $\osr'\simeq\mathbb C$, then $\osr$ does not have the weak expectation property.
\end{proposition}

\begin{proof} Under the stated hypothesis, $\osr^{dd}=\osr$ and $\osi(\osr)=\M_n(\mathbb C)$, with 
the canonical embedding of $\osr$ into $\osr^{dd}$ being the identity map on $\osr$. Assume a ucp extension 
$\phi:\osi(\osr)\rightarrow \osr^{dd}$ of the identity on $\osr$ exists; then, by considering $\osr^{dd}$
as an operator subsystem of $\M_n(\mathbb C)$, $\phi$ is a ucp 
$\phi:\M_n(\mathbb C)\rightarrow \M_n(\mathbb C)$ that fixes $\osr$. Thus, 
by Arveson's Boundary Theorem, $\phi$ is the identity map.
However, if $\phi$ is the identity map, then its range is $\M_n(\mathbb C)$
rather than $\osr$, which contradicts the assumption that $\phi$ maps onto $\osr^{dd}=\osr$.
\end{proof}

\begin{corollary} For every $n\geq 2$, the operator system
$C(S^1)^{(n)}$ does not the weak expectation property.
\end{corollary}

\begin{corollary}\label{no cp idempotent} If $n\geq 2$ and $\mathcal E:\M_n(\mathbb C)\rightarrow C(S^1)^{(n)}$ is an idempotent
linear transformation, then $\mathcal E$ is not completely positive. 
\end{corollary} 

\begin{proof} If $\mathcal E$ were completely positive, then $\mathcal E$ would be a 
completely positive extension of the canonical embedding (namely, the identity) of $C(S^1)^{(n)}$ into its
double dual, $C(S^1)^{(n)}$, to the injective envelope, $\M_n(\mathbb C)$, of $C(S^1)^{(n)}$. 
However, this would imply $C(S^1)^{(n)}$ has the weak expectation
property, contrary to Proposition \ref{no-wep1}.
\end{proof} 

To show $C(S^1)_{(n)}$ does not have the weak expectation property, the following lemma is required.

\begin{lemma}\label{lem2} If a finite-dimensional operator system $\osr$ has the weak expectation property, then
$\osr\omax\ost \subseteq_{\rm coi}  \osi(\osr)\omax\ost$, for every operator system $\ost$.
\end{lemma}

\begin{proof} The finite-dimensionality of $\osr$ implies that $\osr^{dd}\simeq\osr$, and so the weak expectation property of
$\osr$ asserts, assuming the canonical inclusion of $\osr$ as an operator subsystem of $\osi(\osr)$, that there is a ucp map
$\phi:\osi(\osr)\rightarrow\osr$ such that $\phi(x)=x$, for every $x\in \osr$.

Select $p\in\mathbb N$ and a $X\in\M_p(\osr\otimes\ost)$ such that 
$X\in \M_p(\osi(\osr)\omax\ost)_+$. 
Thus,
\[
\left(\phi^{(p)}\otimes\mbox{\rm id}_\ost\right)[X] \in \left(\M_p(\osr)\omax\ost\right)_+=\M_p(\osr\omax\ost)_+.
\]
Hence, $X$ is a positive matrix of $\M_p(\osr\omax\ost)$.
\end{proof}

\begin{proposition} The operator system $C(S^1)_{(n)}$ does not have the weak expectation property, for every $n\geq 2$.
\end{proposition}

\begin{proof} Suppose, on the contrary, that $C(S^1)_{(n)}$ has the weak expectation property, for some $n\geq 2$.
In Lemma \ref{lem2}, let $\osr=C(S^1)_{(n)}$ and $\ost=C(S^1)^{(n)}$. Thus, using Lemmas \ref{ie} and \ref{lem2}, 
\[
C(S^1)_{(n)}\omax C(S^1)^{(n)} \subseteq_{\rm coi}  C(\Delta_{S^1})\omax C(S^1)^{(n)} = C(\Delta_{S^1})\omin C(S^1)^{(n)}.
\]
Therefore, if $p\in\mathbb N$ and $X\in\M_p(C(S^1)_{(n)}\otimes C(S^1)^{(n)})$ is such that 
\[
X\in \M_p(C(\Delta_{S^1})\omin C(S^1)^{(n)})_+,
\]
then $X\in \M_p(C(S^1)_{(n)}\omin C(S^1)^{(n)})$, as $\osr_1\omin\osq\subseteq_{\rm coi} \osr_2\omin\osq$
for all operator systems $\osr_2$ and $\osq$, and all operator subsystems $\osr_1\subseteq\osr_2$.
Hence, 
\[
C(S^1)_{(n)}\omin C(S^1)^{(n)} = C(S^1)_{(n)}\omax C(S^1)^{(n)},
\]
in contradiction to Proposition 6.7 of \cite{farenick2021}.
\end{proof}

%%%%%%%%%%%%%%%%%%%%%%%%%%%%%%%
\subsection{Quotient operator systems and purity of truncation}
Drawing on the work in \cite{kavruk--paulsen--todorov--tomforde2013},
if $\phi:\osr\rightarrow\ost$ is a completely positive linear map of operator systems, and if $\mathcal J$ denotes $\ker\phi$,
then the quotient vector space $\osr/\mathcal J$ is an operator system in which $\M_n(\osr/\mathcal J)_+$ is defined to be the
set of all (cosets) $\dot h\in \M_n(\osr/\mathcal J)$ with the property that, for each $\varepsilon>0$, there is a selfadjoint matrix
$k_\varepsilon\in  \M_n(\mathcal J)$ such that $\varepsilon e_\osr + h + k_\varepsilon \in \M_n(\osr)_+$. (This definition makes use of
the canonical identification of $\M_n(\osr/\mathcal J)$ with $\M_n(\osr)/\M_n(\mathcal J)$.) Furthermore, the First Isomorphism Theorem asserts 
there is a completely positive linear map $\dot\phi:\osr/\mathcal J\rightarrow\ost$ such that $\phi=\dot\phi\circ\pi_{\mathcal J}$, where 
$\pi_{\mathcal J}:\osr\rightarrow\ost/\mathcal J$ is the canonical projection onto the quotient vector space. 
The completely positive linear map $\phi$ is said to be a \emph{complete quotient map} if $\dot\phi$ is a complete order isomorphism.

\begin{proposition}\label{emb1} If $n,m\in\mathbb N$ satisfy $n\leq m$,
then the canonical inclusion map
\[
\iota_{n,m}:C(S^1)_{(n)}\rightarrow C(S^1)_{(m)}
\]
is a unital complete order embedding.
\end{proposition}

\begin{proof} It is sufficient to proof that $\iota_{n,m}$ is a complete isometry, as a ucp map is a complete isometry
if and only if it is a complete order embedding. To this end, select $p\in\mathbb N$ and consider
\[
\iota_{n,m}^{(p)}: \M_p\left( C(S^1)_{(n)}\right)\rightarrow \M_p\left( C(S^1)_{(m)}\right).
\]
Regardless of whether $X=[F_{ij}]_{i,j=1}^n$ is considered as an element of $\M_p\left( C(S^1)_{(n)}\right)$ or of
$\M_p\left( C(S^1)_{(m)}\right)$, the norm of $X$ in both operator systems is given by
\[
\max_{z\in S^1}\left\| [F_{ij}(z)]_{i,j=1}^n\right\|,
\]
which implies $\| \iota_{n,m}^{(p)}(X)\|=\|X\|$.
\end{proof}

\begin{corollary}\label{prj1}
If $n,m\in\mathbb N$ satisfy $n\leq m$, and if 
\[
\mathcal q_{m,n}:C(S^1)^{(m)}\rightarrow C(S^1)^{(n)}
\]
is the canonical projection of each matrix in $C(S^1)^{(m)}$ onto its leading $n\times n$ principal submatrix, then
$\mathcal q_{m,n}=\iota_{n,m}^d$, $\mathcal q_{m,n}$ is a complete quotient map, and
\[
C(S^1)^{(n)} \simeq C(S^1)^{(m)}/\ker\mathcal q_{m,n}
\]
in the operator system  category $\mathfrak S_1$.
\end{corollary}

\begin{proof} By Proposition \ref{emb1}, 
$\iota_{n,m}:C(S^1)_{(n)}\rightarrow C(S^1)_{(m)}$ is a unital complete order embedding; hence, the dual map
$\iota_{n,m}^d:(C(S^1)_{(m)})^d\rightarrow (C(S^1)_{(n)})^d$ is a complete quotient map
\cite[Proposition 1.15]{farenick--paulsen2012}. Thus, in the category $\mathfrak S_1$, 
\[
(C(S^1)_{(n)})^d \simeq (C(S^1)_{(m)})^d / \ker\iota_{n,m}^d.
\]
The Duality Theorem (Theorem \ref{toeplitz duality}) identifies each $(C(S^1)_{q})^d$ with $C(S^1)^{q}$; hence, we
need only show the identification of $\mathcal q_{m,n}$ and $\iota_{n,m}^d$.
To this end,
first note that if $\varphi$ is a linear functional on $C(S^1)_{(m)}$, then it is also a linear functional on
$C(S^1)_{(n)}$ and the adjoint $\iota_{n,m}^d$ of the inclusion map $\iota_{n,m}$ simply sends $\varphi$ to the
functional on $C(S^1)_{(n)}$ whose action on $f\in C(S^1)_{(n)}$ is given by $\varphi(f)$.
The Duality Theorem (Theorem \ref{toeplitz duality}) identifies a linear functional $\varphi$ on $C(S^1)_{(m)}$
with the Toeplitz matrix $t_\varphi=[\tau_{k-j}]$, where $\tau_\ell=\varphi(\chi_{-\ell})$, for $\ell=-m+1,\dots,m-1$, and 
identifies the linear functional $\iota_{n,m}^d(\varphi)$ on $C(S^1)_{(n)}$
with the Toeplitz matrix $t_{\iota_{n,m}^d(\varphi)}=[\tau_{k-j}]$, where $\tau_\ell=\varphi(\chi_{-\ell})$, for $\ell=-n+1,\dots,n-1$.
Hence, after the identification of $(C(S^1)_{q})^d$ with $C(S^1)^{q}$ via the Duality Theorem, the adjoint map
$\iota_{n,m}^d$ is given by $\mathcal q_{m,n}$.
\end{proof}

\begin{corollary}\label{rdc} The ucp map 
\[
\mathcal q_{m,n}\otimes \mathcal q_{m,n}: C(S^1)^{(m)}\omax C(S^1)^{(m)}\rightarrow C(S^1)^{(n)}\omax C(S^1)^{(n)}
\]
is a complete quotient map.
\end{corollary}

\begin{proof} If a ucp map $\phi:\osr\rightarrow\ost$ is a complete quotient map, then so is the ucp map
$\phi\otimes\phi:\osr\omax\osr\rightarrow\ost\omax\ost$  \cite[Corollary 1.13]{farenick--paulsen2012}.
\end{proof} 

Observe that Corollary \ref{rdc} above is a stronger assertion than Theorem \ref{max is sep} in that the latter result is in reference to the
base cone, while the former is in reference to all matrix cones. These ideas apply to the problem of extending positive block-Toeplitz matrices to 
larger positive block Toeplitz matrices.

\begin{theorem}\label{ext-nucl} If $\osr$ is a nuclear operator system, and if 
$x_0,x_1, \dots, x_{n-1}\in\osr$ are such that the matrix
\[
X=\left[ \begin{array}{cccc} x_0 & x_{1}^* & \dots & x_{n-1}^* \\
x_1 & x_0 & \ddots & \vdots \\ \vdots & \ddots & \ddots & x_1^* \\
x_{n-1} & \dots & x_1 & x_0 
\end{array}
\right]
\]
is strictly positive, then for each $m>n$ there exists $x_n,\dots,x_{m-1}\in\osr$ such that
the matrix
\[
\tilde X= 
\left[ \begin{array}{cccccc} x_0 & x_{1}^* & \dots & x_{n-1}^* &\dots & x_{m-1}^* \\
x_1 & x_0 & \ddots &&& \vdots  \\  
\vdots &\ddots&\ddots&\ddots & & \\
&&&\ddots&\ddots&\vdots \\
\vdots &   &   & & x_0& x_1^* \\
x_{m-1} & \dots &x_{n-1} & \dots & x_1 & x_0 
\end{array}
\right]
\]
is strictly positive.
\end{theorem}

\begin{proof} Assume $X\in \left( C(S^1)^{(n)}\omin \osr\right)_+$ is strictly positive, and fix $m>n$. 
Since $\osr$ is nuclear, the matrix $X$ belongs to the positive cone of $C(S^1)^{(n)}\omax \osr$, and 
because the truncation map 
$\mathcal q_{m,n}\otimes\mbox{\rm Id}_\osr: C(S^1)^{(m)}\omax \osr\rightarrow C(S^1)^{(n)}\omax \osr$ is a complete quotient map
\cite[Proposition 1.12]{farenick--paulsen2012}, the strictly positive matrix $X$ is of the form $(\mathcal q_{m,n}\otimes\mbox{\rm Id}_\osr)[\tilde X]$,
for some strictly positive matrix $\tilde X\in C(S^1)^{(m)}\omax \osr$ \cite[Proposition 3.2]{farenick--kavruk--paulsen2013}.
\end{proof} 

Theorem \ref{ext-nucl} does not extend in general to non-nuclear operator systems, in light of the observations in \cite{ozawa2013}
related to Rudin's work \cite{rudin1963}. 

Lastly, we note the embeddings of Fej\'er-Riesz operator systems and truncations of Toeplitz operator systems are extremal
completely positive linear maps. Recall that a completely positive linear map $\phi:\osr\rightarrow\ost$ of operator systems is
\emph{pure}, if the only completely positive linear maps $\vartheta,\psi:\osr\rightarrow\ost$ that satisfy $\vartheta+\psi=\phi$ are
those in which $\vartheta$ and $\psi$ are of the form $\vartheta=s\phi$ and $\psi=(1-s)\phi$ for some real number $s\in[0,1]$.

\begin{proposition} The ucp maps $\iota_{n,m}$ and $\mathcal q_{m,n}$ are pure, for all $2\leq n\leq m$.
\end{proposition}

\begin{proof} Because $\mathcal q_{m,n}=\iota_{n,m}^d$, it is enough to prove that $\mathcal q_{m,n}$ is
pure \cite[Proposition 2.6]{farenick--tessier2022}. To this end, suppose  $\vartheta,\psi:C(S^1)^{(m)}\rightarrow C(S^1)^{(n)}$ 
are completely positive linear
maps such that $\vartheta+\psi=\mathcal q_{m,n}$. For each $\lambda\in S^1$, 
$T_n(\lambda)=\mathcal q_{m,n}\left(T_m(\lambda)\right)$ and, thus, 
$T_n(\lambda)=\vartheta(T_m(\lambda))+\psi(T_m(\lambda))$. Because $T_n(\lambda)$ is a pure element of
$C(S^1)^{(n)}$ for every $\lambda\in S^1$ \cite[Proposition 4.8]{connes-vansuijlekom2021}, there are scalars
$s_\lambda\in[0,1]$ such that $\vartheta(T_m(\lambda))=s_\lambda T_n(\lambda)$ and 
$\psi(T_m(\lambda))=(1-s_\lambda) T_n(\lambda)$. We first show that the value of $s_\lambda$ is independent of $\lambda$.

Let $a=\vartheta( 2\cdot 1_m)$ and let $\alpha$ denote the diagonal entry of $a$ and $\beta$ denote the subdiagonal entry of $a$.
For each $\lambda\in S^1$, the matrix $T_m(\lambda)+T_m(-\lambda)$ is twice the identity, and so 
\[
a =\vartheta (T_m(\lambda)+T_m(-\lambda)) = \vartheta (T_m(\lambda))+\vartheta (T_m(\lambda))
=s_\lambda T_n(\lambda)+s_{-\lambda}T_n(-\lambda).
\]
Equating the (1,1) and (2,1) matrix entries of $a$ and $s_\lambda T_n(\lambda)+s_{-\lambda}T_n(-\lambda)$ leads to 
the linear system of equations given by
\[
\begin{array}{rcl}
\alpha &=& s_\lambda+s_{-\lambda} \\
\frac{\beta}{\lambda}&=& s_\lambda-s_{-\lambda}.
\end{array}
\]
Thus, $\alpha+\frac{\beta}{\lambda}=2s_\lambda$, which implies that $\frac{\beta}{\lambda}\in\mathbb R$ for 
every $\lambda\in S^1$. Hence, $\beta$ must be zero, which implies $s_\lambda= \frac{\alpha}{2}$, for every 
$\lambda\in S^1$.
Thus, if $s=\frac{\alpha}{2}$, then $\vartheta( T_m(\lambda))=s T_n(\lambda)$
for every $\lambda\in S^1$.

Now let $x$ be an arbitrary element of the positive cone of $C(S^1)^{(m)}$. As $x$ is a sum of pure elements, 
there exist $k\in\mathbb N$, scalars $\alpha_j\in\mathbb R_+$, and $\lambda_j\in S^1$ such that
$x=\displaystyle\sum_{j=1}^k \alpha_j T_m(\lambda_j)$. Thus, 
\[
\mathcal q_{m,n}(x) = \sum_{j=1}^k \alpha_j T_n(\lambda_j) \mbox{ and }
\vartheta (x) = \sum_{j=1}^k \alpha_j sT_n(\lambda_j)=s\mathcal q_{m,n}(x).
\]
Because $\vartheta$ and $s\mathcal q_{m,n}$ agree on the positive cone of $C(S^1)^{(m)}$, which spans
$C(S^1)^{(m)}$, $\vartheta=s\mathcal q_{m,n}$ and $\psi=(1-s)\mathcal q_{m,n}$,
which proves that $\mathcal q_{m,n}$ is pure.
\end{proof}

As a coda to these results above on purity, let us make note of the following theorem.

\begin{theorem} If $n\geq 2$, then
\begin{enumerate}
\item the canonical ucp embedding of $C(S^1)^{(n)}$ into $\M_n(\mathbb C)$ is pure in the cone of all completely linear maps 
$C(S^1)^{(n)}\rightarrow\M_n(\mathbb C)$, and 
\item the canonical ucp embedding of $C(S^1)_{(n)}$ into $C(S^1)$ is not pure in the cone of all completely linear maps 
$C(S^1)_{(n)}\rightarrow C(S^1)$.
\end{enumerate}
\end{theorem}

\begin{proof} 
Because the C$^*$- and injective envelope of $C(S^1)^{(n)}$ coincide with the simple C$^*$-algebra
$\M_n(\mathbb C)$, statement (1) follows from \cite[Theorem 3.2]{farenick--tessier2022}. For statement (2),
the fact that $C(S^1)$ is a commutative C$^*$-algebra
implies the canonical ucp embedding of $C(S^1)_{(n)}$ into $C(S^1)$ is not pure in the cone of all completely linear maps 
$C(S^1)_{(n)}\rightarrow C(S^1)$ \cite[Proposition 2.14]{farenick--tessier2022}.
\end{proof}

%%%%%%%%%%%%%%%%%%%
\section*{Acknowledgement}

I wish to thank N.~Ozawa for informing me of references \cite{ozawa2013,rudin1963}, which show that
the extension theorem for 
$2\times 2$ positive block-Toeplitz matrices with Toeplitz blocks (Corollary \ref{extension}) does not extend to 
$n\times n$ positive block-Toeplitz matrices with Toeplitz blocks, for $n\geq3$.
This work has been supported, in part, by a 
Natural Sciences and Engineering Research Council of Canada
Discovery Grant.

 %%%%%%%%%%%%%%%%%%%%%%%

\end{document}